\documentclass[10pt]{amsart}

\usepackage{amsthm, amsfonts, amssymb, color}
 \usepackage{mathrsfs}
\usepackage{amsmath}
 \usepackage{amstext, amsxtra}
  \usepackage{txfonts}
 \usepackage[colorlinks, linkcolor=black, citecolor=blue, pagebackref, hypertexnames=false]{hyperref}
\usepackage{graphicx}
\usepackage{overpic}

 \allowdisplaybreaks
\newtheorem{theorem}{Theorem}[section]
\newtheorem{proposition}[theorem]{Proposition}

\newtheorem{lemma}[theorem]{Lemma}
\newtheorem{definition}[theorem]{Definition}

\newtheorem{remark}[theorem]{Remark}

\renewcommand\Re{\operatorname{Re}}
\renewcommand\Im{\operatorname{Im}}

\title[Dynamics of NLS with a potential ]{
Dynamics of the focusing 3D cubic NLS with slowly decaying potential}
\author{Qing Guo, \ Hua Wang \ and Xiaohua Yao}

\address {Qing Guo, College of Science,
 Minzu University of China, Beijing, 100081, P.R. China}
\email{guoqing0117@163.com}
\address{ Hua Wang, School of Mathematics and Statistics and Hubei Province Key Laboratory of Mathematical Physics,
Central China Normal University, Wuhan, 430079, P.R. China}
\email{wanghua\_math@126.com}
\address{Xiaohua Yao, School of Mathematics and Statistics and  Hubei Province Key Laboratory of Mathematical Physics,
 Central China Normal University, Wuhan, 430079, P.R. China}
\email{yaoxiaohua@mail.ccnu.edu.cn }
\date{\today}

\begin{document}

\maketitle
\begin{abstract}
In this paper, we consider a 3d cubic focusing nonlinear schr\"{o}dinger equation (NLS) with slowing decaying potentials.
Adopting the variational method of Ibrahim-Masmoudi-Nakanishi \cite{IMN}, we obtain a condition for scattering. It is
actually sharp in some sense since the solution will blow up if it's false. The proof of blow-up part relies on the
method of Du-Wu-Zhang \cite{DWZ}.
\end{abstract}

\tableofcontents

\section{Introduction}
\setcounter{equation}{0}
In this paper, we consider a 3d cubic focusing NLS with slowly decaying potentials ($\rm{NLS_{k}}$)
\begin{equation}\label{1.1}
\left\{ \begin{aligned}
  i&\partial_{t}u-H_{\alpha}u+|u|^{2}u=0,\;\;(t,x) \in {{\bf{R}}\times{\bf{R}}^{3}}, \\
  u&(0, x)=u_{0}(x)\in H^{1}({\bf{R}}^{3}),
 \end{aligned}\right.
 \end{equation}
where $u: {\bf R}\times {\bf R}^{3}\rightarrow {\bf C}$ is a complex-valued function, $H_{\alpha}=-\Delta+V(x)$ and
$V(x)=\frac{k}{|x|^{\alpha}}$ with $k>0$ and $1<\alpha\leq 2$. Throughout this paper, we use the symbol $V(x)$
instead of $\frac{k}{|x|^{\alpha}}$ since we frequently use the general property of $V$: $V\geq 0$, $x\cdot\nabla V\leq 0$,
$2V+x\cdot V\geq 0$ and $3x\cdot V+x\nabla^{2}V x^{T}\leq 0$. As $\frac{k}{|x|^{\alpha}}>0$ and $\frac{k}{|x|^{\alpha}}
\in L_{loc}^{1}$, $H_{\alpha}$ is defined as a unique self-adjoint operator associated with the non-negative quadratic
form $<(-\Delta+\frac{k}{|x|^{\alpha}})f, f>$ on $C_{0}^{\infty}({\bf R}^{3})$. Moreover, $H_{\alpha}$ is purely absolutely
continuous and has no eigenvalue. Since $k>0$,
the kernel $e^{-tH_{\alpha}}(x, y)$ of $e^{-tH_{\alpha}}$ satisfies
the upper Gaussian estimate \cite{S} i.e., for $\forall t>0$, $\forall x$, $y\in {\bf R}^{3}$,
\begin{align}\label{heatkernel}
0\leq e^{-tH_{\alpha}}(x, y)\leq e^{t\Delta}(x, y)=(4\pi t)^{-\frac{3}{2}}e^{-\frac{|x-y|^{2}}{4t}},
\end{align}
which implies that Hardy inequality \eqref{Hardy}, Mikhlin multiplier theorem and Littlewood-Paley theory (Bernstein inequalities Lemma \ref{Bernstein},
Littlewood-Paley decomposition Lemma \ref{LPdecompostion}
and square function estimates Lemma \ref{square}) associated with $H_{\alpha}$. Hence it follows from Hardy inequality and Stein complex interpolation that the
standard Sobolev norms and the Sobolev norms associated with $H_{\alpha}$ are equivalent (see Lemma \ref{Sobolev}). Recently, Mizutani \cite{M} showed that $e^{-itH_{\alpha}}$ satisfies
global-in-time Strichartz estimates for any admissible pairs. Combining the Sobolev norm equivalence and the Strichartz estimates and following
the same line of the proof of Theorem 2.15 and Remark 2.16 of \cite{KMVZ} yield that ($\rm{NLS_{k}}$) is locally well-posed and scatters in $H^{1}({\bf R}^{3})$.

\begin{theorem} \cite{KMVZ} \label{localwellposedness}
Let $u_{0}\in H^{1}({\bf R}^{3})$. Then the following are true.

(i) There exist $T=T(\|u_{0}\|_{H^{1}})>0$ and a unique solution $u\in C((-T, T), H^{1}({\bf R}^{3}))$ of ($\rm{NLS_{k}}$).

(ii) There exists $\epsilon_{0}>0$ such that if for $0<\epsilon<\epsilon_{0}$,
$$
\|e^{-itH_{\alpha}}u_{0}\|_{L_{t,x}^{5}({\bf R}^{+}\times{\bf R}^{3})}<\epsilon,
$$
then the solution $u$ of ($\rm{NLS_{k}}$) is global in the positive time direction and satisfies
\begin{align*}
\|u\|_{L_{t,x}^{5}({\bf R}^{+}\times{\bf R}^{3})}\lesssim\epsilon.
\end{align*}
 The similar result holds in the negative time direction.

(iii) For any $\phi\in H^{1}({\bf R}^{3})$, then there exist $T>0$ and a solution $u\in C((T, +\infty), H^{1}({\bf R}^{3}))$ of ($\rm{NLS_{k}}$)
such that
$$
\lim_{t\rightarrow+\infty}\|u(t)-e^{-itH_{\alpha}}\phi\|_{H^{1}({\bf R}^{3})}=0.
$$
The similar result holds in the negative time direction.

(iv) If $u:{\bf R}\times{\bf R}^{3}\rightarrow {\bf C}$ is a global solution of ($\rm{NLS_{k}}$) with
\begin{align}\label{scatteringbound}
\|u\|_{L_{t,x}^{5}({\bf R}\times{\bf R}^{3})}<+\infty,
\end{align}
then the solution $u(t)$ scatters in $H^{1}$. That is, there exists $\phi_{\pm}\in H^{1}({\bf R}^{3})$ such that
$$
\lim_{t\rightarrow\pm\infty}\|u(t)-e^{-itH_{\alpha}}\phi_{\pm}\|_{H^{1}({\bf R}^{3})}=0.
$$

\end{theorem}

Moreover, the $H^{1}$ solution $u$ obeys the mass and energy conservation laws:
\begin{align}\label{mass}
M(u)=\displaystyle\int_{{\bf R}^{3}}|u(t, x)|^{2}dx=M(u_{0}),
\end{align}
and
\begin{align}\label{energy}
E(u)=E_{k}(u)=\frac{1}{2}\displaystyle\int_{{\bf R}^{3}}|\nabla u(x)|^{2}dx
+\frac{1}{2}\displaystyle\int_{{\bf R}^{3}}V(x)|u(x)|^{2}dx
-\frac{1}{4}\displaystyle\int_{{\bf R}^{3}}|u(x)|^{4}dx=E(u_{0}).
\end{align}

In the case $k=0$, Holmer-Roudenko \cite{HR} and Duyckaerts-Holmer-Roudenko \cite{DHR} employed the concentration-compactness
approach of Kenig-Merle \cite{KM} to obtain sharp criteria between scattering and blow up for ($\rm{NLS_{0}}$) in terms of
conservation laws (\eqref{mass} and \eqref{energy}) and the ground state $Q$, which is the unique positive radial exponential decaying
solution of the elliptic equation
\begin{align}\label{ellipticequation}
\Delta Q-Q+Q^{3}=0.
\end{align}
Fang-Xie-Canzenave \cite{FXC} and Akahor-Nawa \cite{AN} extended the result in \cite{HR, DHR} to the general power and dimensions. Subsequently,
Killip-Murphy-Visan-Zhang  \cite{KMVZ} established
a corresponding sharp threshold between scattering and blow up for ($\rm{NLS_{k}}$) with $k>-\frac{1}{4}$ and $\alpha=2$. Recently, Miao-Zhang-Zheng \cite{MZZ} used
the interaction Morawetz-type estimates and the equivalence of Sobolev norms to prove all solutions scatter for ($\rm{NLS_{k}}$) with $k>0$, $\alpha=1$ and
$-|u|^{p-1}u$ $(\frac{7}{3}<p<5)$ in place of $|u|^{2}u$ (i.e., nonlinear Schr\"{o}dinger equation with repulsive Coulomb potential in the defocusing).

The goal of this paper is to extend the sharp scattering criterion in \cite{KMVZ} from $\alpha=2$ to $1<\alpha\leq 2$ when $k>0$ in some sense. Obviously, for $1<\alpha<2$,
the equation ($\rm{NLS_{k}}$) doesn't enjoy scaling invariant. Therefore, we cannot apply scaling as indicated in \cite{KMVZ} to get a critical element
( a minimal blow up solution). Hence, we
shall adopt the variational argument based on the work of Ibrahim-Masmoudi-Nakanishi \cite{IMN} to overcome the difficulty. Recently, the same argument have been applied to
the focusing mass-supercritical nonlinear Schr\"{o}dinger equation with repulsive Dirac delta potential on the real line (see \cite{II}).

To state our main result, we introduce some notation now. We define the functional $S_{k}$ as
\begin{align}\label{action}
S_{k}(\varphi):=E(\varphi)+\frac{1}{2}M(\varphi)=\frac{1}{2}\|\varphi\|_{{\mathcal H}_{k}^{1}}-\frac{1}{4}\displaystyle\int_{{\bf R}^{3}}|\varphi(x)|^{4}dx,
\end{align}
where
\begin{align}\label{Sobolev1}
\|\varphi\|_{{\mathcal H}_{k}^{1}}^{2}=\displaystyle\int_{{\bf R}^{3}}|\nabla \varphi(x)|^{2}dx
+\displaystyle\int_{{\bf R}^{3}}V(x)|\varphi(x)|^{2}dx
+\displaystyle\int_{{\bf R}^{3}}|\varphi(x)|^{2}dx,
\end{align}
which is equivalent to $\|\varphi\|_{H^{1}}$ by $k>0$ by Hardy's inequality.
Denote the scaling quantity $\varphi_{\lambda}^{a, b}$ by
\begin{align}\label{scalingquantity}
\varphi_{\lambda}^{a, b}:=e^{a\lambda}\varphi(e^{-b\lambda}x),
\end{align}
where $(a, b)$ satisfies the condition
\begin{align}\label{parameter}
a>0, \;\; b\leq 0,\;\; 2a+b> 0\;\; 2a+3b\geq 0,\;\; (a,b)\neq (0,0).
\end{align}
We define the scaling derivative of $S_{k}(\varphi_{\lambda}^{a, b})$ at $\lambda=0$ by $K_{k}^{a, b}(\varphi)$.
\begin{align}\label{functionalK}
K_{k}^{a, b}(\varphi)&:={\mathcal{L}}^{a,b}S_{k}(\varphi)=\frac{d}{d\lambda}\Big{|}_{\lambda=0}S_{k}(\varphi_{\lambda}^{a, b})\nonumber\\
&=\frac{2a+b}{2}\displaystyle\int_{{\bf R}^{3}}|\nabla \varphi(x)|^{2}dx
+\frac{2a+3b}{2}\displaystyle\int_{{\bf R}^{3}}V|\varphi(x)|^{2}dx
+\frac{b}{2}\displaystyle\int_{{\bf R}^{3}}(x\cdot\nabla V) |\varphi(x)|^{2}dx\\
&\;\;+\frac{2a+3b}{2}\displaystyle\int_{{\bf R}^{3}}|\varphi(x)|^{2}dx
-\frac{4a+3b}{4}\displaystyle\int_{{\bf R}^{3}}|\varphi(x)|^{4}dx\nonumber
\end{align}
In particular, when $(a, b)=(3, -2)$,
\begin{align}\label{functionalP}
P_{k}(\varphi):=K_{k}^{3,-2}(\varphi)=2\displaystyle\int_{{\bf R}^{3}}|\nabla \varphi(x)|^{2}dx
-\displaystyle\int_{{\bf R}^{3}}(x\cdot\nabla V) |\varphi(x)|^{2}dx
-\frac{3}{2}\displaystyle\int_{{\bf R}^{3}}|\varphi(x)|^{4}dx,
\end{align}
which is related with the Virial identity of ($\rm{NLS_{k}}$) \eqref{I''0} with $\phi(x)=|x|^{2}$, and when
$(a, b)=(3, 0)$,
\begin{align}\label{functionalI}
I_{k}(\varphi):=\frac{1}{3}K_{k}^{3,0}(\varphi)=\displaystyle\int_{{\bf R}^{3}}|\nabla \varphi(x)|^{2}dx
+\displaystyle\int_{{\bf R}^{3}}V|\varphi(x)|^{2}dx
+\displaystyle\int_{{\bf R}^{3}}|\varphi(x)|^{2}dx
-\displaystyle\int_{{\bf R}^{3}}|\varphi(x)|^{4}dx.
\end{align}
We note that, to get existence of minimal blow-up solutions, we need to use the functional $I_{k}$ instead of $P_{k}$.
so that we can apply the linear profile decomposition Lemma \ref{linearprofile}.

The sharp threshold quantity $n_{k}$ are determined by the following minimizing problem
\begin{align}\label{threshold}
n_{k}=\inf\{S_{k}(\varphi):\varphi\in H^{1}({\bf R}^{3})\setminus \{0\}, P_{k}(\varphi)=0\}.
\end{align}
When $k=0$, $n_{0}$ is positive and is achieved by $Q$, which is a unique radial solution of \eqref{ellipticequation}
(see \cite{AN}). The sharp criteria between scattering and blow up as mentioned above can be described by $n_{0}$.
We state the results of \cite{HR, DHR, FXC, AN} in terms of $n_{0}$ as follows.

\begin{theorem} \cite{HR, DHR, FXC, AN} \label{scattering0}
Let $u_{0}\in H^{1}({\bf R}^{3})$ satisfy $S_{0}(u_{0})<n_{0}$.

(i) If $P_{0}\geq 0$, then the solution $u$ of ($\rm{NLS_{0}}$) is global and scatters.

(ii) If $P_{0}< 0$ and $u_{0}$ is radial or $xu_{0}\in L^{2}({\bf R}^{3})$, the the solution $u$ of ($\rm{NLS_{0}}$) blows up in
finite time in both time directions.

Furthermore, if $\psi\in H^{1}({\bf R}^{3})$ satisfying $\frac{1}{2}\|\psi\|_{H^{1}}^{2}<n_{0}$, then there exists a
global solution of ($\rm{INLS_{0}}$) that scatters to $\psi$ in the positive time direction.
The analogous statement holds in the negative time direction.
\end{theorem}

When $k>0$, we prove that $n_{k}=n_{0}$ and $n_{k}$ is never attained (see Lemma \ref{attain}). For succinctness, we next define two subsets of
$H^{1}({\bf R}^{3})$ as follows:
\begin{align}\label{n+}
{\mathcal{N}}^{+}:=\{\varphi\in H^{1}({\bf R}^{3}): S_{k}(\varphi)<n_{0}, P_{k}(\varphi)\geq 0\}
\end{align}
and
\begin{align}\label{n-}
{\mathcal{N}}^{-}:=\{\varphi\in H^{1}({\bf R}^{3}): S_{k}(\varphi)<n_{0}, P_{k}(\varphi)< 0\}.
\end{align}
Now we state our main result.

\begin{theorem}\label{scattering1}
Let $u$ be the solution of ($\rm{NLS_{k}}$) on $(-T_{min}, T_{max})$, where $(-T_{min}, T_{max})$ is the maximal life-span.

(i) If $u_{0}\in {\mathcal{N}}^{+}$, then $u$ is global well-posedness, $u(t)\in {\mathcal{N}}^{+}$ for any $t\in {\bf R}$ and scatters.

(ii) If $u_{0}\in {\mathcal{N}}^{-}$, then $u(t)\in {\mathcal{N}}^{-}$ for any $t\in (-T_{min}, T_{max})$ and one of the following four statements holds true:

\;\;(a) $T_{max}<+\infty$ and $\lim_{t\uparrow T_{max}}\|\nabla u(t)\|_{L^{2}}=+\infty$.

\;\;(b) $T_{min}<+\infty$ and $\lim_{t\downarrow -T_{min}}\|\nabla u(t)\|_{L^{2}}=+\infty$.

\;\;(c) $T_{max}=+\infty$ and there exists a sequence $\{t_{n}\}_{n=1}^{+\infty}$ such that $t_{n}\rightarrow+\infty$ and $\lim_{t_{n}\uparrow +\infty}\|\nabla u(t)\|_{L^{2}}=+\infty$.

\;\;(d) $T_{min}=+\infty$ and there exists a sequence $\{t_{n}\}_{n=1}^{+\infty}$ such that $t_{n}\rightarrow-\infty$ and $\lim_{t_{n}\downarrow -\infty}\|\nabla u(t)\|_{L^{2}}=+\infty$.
\end{theorem}

Here the blow-up result is proved by the method of Du-Wu-Zhang \cite{DWZ}.

This present paper is organized as follows. We fix notations at the end of Section 1. In Section 2, as preliminaries, we state some our required lemmas, including
Sobolev norm equivalence, Strichartz estimates, stability theory, Littlewood-Paley theory, some limit lemmas between $H_{\alpha}^{n}$ and $H_{\alpha}^{\infty}$,
linear profile decomposition and nonlinear profiles for $|x_{n}|\rightarrow+\infty$. In Section 3, using the variational idea of Ibrahim-Masmoudi-Nakanishi \cite{IMN},
we obtain that if $\psi\in \mathcal{N}^{+}$, $P_{k}(\psi)$ and $I_{k}(\psi)$ have the same sign and $S_{k}(\psi)$ is equivalent to $\|\psi\|_{H^{1}}$ and
if $\psi\in \mathcal{N}^{\pm}$, $P_{k}(\psi)$ has the uniform bounds, which play a vital role to get blow-up and scattering results. In Section 4, using the upper bound
of $P_{k}(\psi)$ for $\psi\in \mathcal{N}^{-}$ and adopting the method of Du-Wu-Zhang \cite{DWZ}, we establish blow-up part of Theorem \ref{scattering1}. Global part of
Theorem \ref{scattering1} can be obtained by the lower bound of $P_{k}(\psi)$ for $\psi\in \mathcal{N}^{+}$ and local well-posedness $i$ of Theorem \ref{localwellposedness}.
In the last section, we show the scattering part of Theorem \ref{scattering1} in two steps. In Step 1, by contradiction, if scattering fails, then a critical element must exist.
In Step 2, we utilize lower bound of $P_{k}(\psi)$ for $\psi\in \mathcal{N}^{+}$ to preclude the critical element. Putting the last two sections together completes the
proof of Theorem \ref{scattering1}.

\textbf{Notations:}:

We fix notations used throughout the paper. In what follows, we write $A\lesssim B$ to signify
that there exists a constant $c$ such that $A\leq cB$, while we denote $A\sim B$ when
$A\lesssim B\lesssim A$. Given a real number $\alpha$, $\alpha-=\alpha-\epsilon$ for $0<\epsilon\ll 1$.

Let $L_{I}^{q}L_{x}^{r}$ be the space
of measurable functions from an interval $I\subset {\bf R}$ to $L_{x}^{r}$ whose $L_{I}^{q}L_{x}^{r}$-
norm $
\|\cdot\|_{L_{I}^{q}L_{x}^{r}}$ is finite, where
\begin{align*}
\|u\|_{L_{I}^{q}L_{x}^{r}}=\Big(\displaystyle\int_{I}\|u(t)\|_{L_{x}^{r}}^{q}dt\Big)^{\frac{1}{r}}.
\end{align*}
When $I={\bf R}$, we may use $L_{t}^{q}L_{x}^{r}$ instead of
$L_{I}^{q}L_{x}^{r}$, respectively. In particular, when $q=r$, we may simply write them as
$L_{t,x}^{q}$, respectively.

Moreover, the Fourier transform on ${\bf R}^{3}$ is defined by
$\hat{f}(\xi)=(2\pi)^{-\frac{3}{2}}\displaystyle\int_{{\bf R}^{3}} e^{-ix\cdot\xi}f(x)dx$.
Define the inhomogeneous Sobolev space $H^{s}({\bf R}^{3})$ and and the homogeneous Sobolev space $\dot{H}^{s}({\bf R}^{3})$, respectively, by norms
$$
\|f\|_{H^{s}({\bf R}^{3})}=\|(1+|\xi|^{2})^{\frac{s}{2}}\hat{f}(\xi)\|_{L^{2}({\bf R}^{3})}=\|(1+\Delta)^{\frac{s}{2}} f\|_{L^{2}({\bf R}^{3})}
$$
and
$$
\|f\|_{\dot{H}^{s}({\bf R}^{3})}=\||\xi|^{s}\hat{f}(\xi)\|_{L^{2}({\bf R}^{3})}=\|\Delta^{\frac{s}{2}} f\|_{L^{2}({\bf R}^{3})}.
$$
Denote the inhomogeneous Sobolev space and homogeneous Sobolev space adapted to $H_{\alpha}$ by ${\mathcal{H}}_{k}^{s, p}({\bf R}^{3})$
and ${\mathcal{\dot H}}_{k}^{s, p}({\bf R}^{3})$, respectively, with norms
$$
\|f\|_{{\mathcal{H}}_{k}^{s, p}({\bf R}^{3})}=\|(1+H_{\alpha})^{\frac{s}{2}} f\|_{L^{p}({\bf R}^{3})}
$$
and
$$
\|f\|_{{\mathcal{\dot H}}_{k}^{s, p}({\bf R}^{3})}=\|H_{\alpha}^{\frac{s}{2}} f\|_{L^{p}({\bf R}^{3})}.
$$
${\mathcal{H}}_{k}^{s}({\bf R}^{3})$
and ${\mathcal{\dot H}}_{k}^{s}({\bf R}^{3})$ are shorthand by ${\mathcal{H}}_{k}^{s, 2}({\bf R}^{3})$
and ${\mathcal{\dot H}}_{k}^{s, 2}({\bf R}^{3})$, respectively.

Given $p\geq 1$, let $p'$ be the conjugate of $p$, that is $\frac{1}{p}+\frac{1}{p'}=1$.

{\bf Acknowledgement}\ \ The first author is financially supported by the
China National Science Foundation (No.11301564, 11771469), the second author is
financially supported by the
China National Science Foundation ( No. 11771165 and
11571131), and the third author is financially supported by the
China National Science Foundation( No. 11771165).



\section{Preliminaries}\label{sec-2}
\setcounter{equation}{0}
As we've mentioned in the introduction, the heat kernel associated with $H_{\alpha}$ satisfies \eqref{heatkernel}, so Mikhlin multiplier theorem holds, which implies
that for $\forall 1<p<\infty$,
\begin{align}\label{Lpbound}
\|f\|_{L^{p}}\lesssim \|(1+H_{\alpha})f\|_{L^{p}}.
\end{align}
And we have the following Hardy type inequality for $H_{\alpha}$ (e.g., see \cite{KMVZS} for $\alpha=2$).
\begin{align}\label{Hardy}
\Big\||x|^{-s}f\Big\|_{L^{p}({\bf R}^{3})}\lesssim \|H_{\alpha}^{\frac{s}{2}}f\|_{L^{p}({\bf R}^{3})}\lesssim \|(1+H_{\alpha})^{\frac{s}{2}}f\|_{L^{p}({\bf R}^{3})},
\end{align}
where $0<s<3$ and $1<p<\frac{3}{s}$.

Using \eqref{Lpbound} and Stein complex interpolation yields the following Sobolev norm equivalence
(see \cite{H} for $V\geq 0$ and $V\in L^{\frac{3}{2}}$ and \cite{ZZ, KMVZS, MZZ} for $\alpha=2$).

\begin{lemma}  \label{Sobolev}
Let $k>0$ and $0<\alpha< 2$, $1<p<\frac{3}{s}$ and $0\leq s\leq 2$, then
\begin{align}\label{ihomogeneous}
\|(1+H_{\alpha})^{\frac{s}{2}}f\|_{L^{p}({\bf R}^{3})}\sim \|(1-\Delta)^{\frac{s}{2}}f\|_{L^{p}({\bf R}^{3})}.
\end{align}
\end{lemma}

\begin{proof} As the heat kernel associated with $H_{\alpha}$ satisfies \eqref{heatkernel}, the kernel of Riesz potentials $(1+H_{\alpha})^{-\frac{s}{2}}$
$$
(1+H_{\alpha})^{-\frac{s}{2}}(x, y)=\frac{1}{\Gamma(\frac{s}{2})}\displaystyle\int_{0}^{+\infty}e^{-t(1+H_{\alpha})}(x, y)t^{\frac{s}{2}-1}dt
$$
satisfies
$$
|(1+H_{\alpha})^{-\frac{s}{2}}(x, y)|\lesssim |x-y|^{s-3},
$$
which, by Hardy-Littlewood-Sobolev inequality, implies that
\begin{align}\label{HLS}
\|(1+H_{\alpha})^{-\frac{s}{2}}f\|_{L^{\frac{3p}{3-ps}}}\lesssim \|f\|_{L^{p}}.
\end{align}

It suffices to prove that \eqref{ihomogeneous} with $s=2$ and $1<p<\frac{3}{2}$ holds. Indeed, if it is true, then \eqref{ihomogeneous}
follows from Stein complex interpolation and the $L^{p}$-boundedness of $(1+H_{\alpha})^{iy}$ with $\forall y\in {\bf R}$ and $1<p<+\infty$
(which can be obtained by \eqref{heatkernel} and Sikora-Wright \cite{SW}).

Let $\chi(x)$ be a smooth compact supported function such that
$\chi(x)=1$ for $|x|\leq 1$ and $\chi(x)=0$ for $|x|\geq 2$. On one hand, using H\"{o}lder inequality and Sobolev embedding yields that
\begin{align*}
\|(1+H_{\alpha})f\|_{L^{p}}&\leq \|(1-\Delta)f\|_{L^{p}}+k\Big\|\frac{1}{|x|^{\alpha}}f\Big\|_{L^{p}}\nonumber\\
&\leq \|(1-\Delta)f\|_{L^{p}}+k\Big\|\frac{1}{|x|^{\alpha}}\chi f\Big\|_{L^{p}}+k\Big\|\frac{1}{|x|^{\alpha}}(1-\chi) f\Big\|_{L^{p}}\nonumber\\
&\lesssim \|(1-\Delta)f\|_{L^{p}}+\Big\||x|^{-\alpha}\chi\Big\|_{L^{\frac{3}{2}}}\|f\|_{L^\frac{3p}{3-2p}}+\|f\|_{L^{p}}\nonumber\\
&\lesssim \|(1-\Delta)f\|_{L^{p}}.
\end{align*}
On the other hand, using H\"{o}lder inequality, \eqref{Lpbound} and \eqref{HLS} with $s=2$ gives that
\begin{align*}
\|(1-\Delta)f\|_{L^{p}}&\leq \|(1+H_{\alpha})f\|_{L^{p}}+k\Big\|\frac{1}{|x|^{\alpha}}f\Big\|_{L^{p}}\nonumber\\
&\leq \|(1+H_{\alpha})f\|_{L^{p}}+k\Big\|\frac{1}{|x|^{\alpha}}\chi f\Big\|_{L^{p}}+k\Big\|\frac{1}{|x|^{\alpha}}(1-\chi) f\Big\|_{L^{p}}\nonumber\\
&\lesssim \|(1+H_{\alpha})f\|_{L^{p}}+\Big\||x|^{-\alpha}\chi\Big\|_{L^{\frac{3}{2}}}\|f\|_{L^\frac{3p}{3-2p}}+\|f\|_{L^{p}}\nonumber\\
&\lesssim \|(1+H_{\alpha})f\|_{L^{p}}.
\end{align*}
Thus, we get \eqref{ihomogeneous} with $s=2$ and $1<p<\frac{3}{2}$ and then conclude the proof.

\end{proof}

Recently, Mizutani \cite{M} proved that the solution to free Schr\"{o}dinger equations with a class of slowly decaying repulsive
potentials including $k|x|^{-\alpha}$ with $k>0$ and $0<\alpha<2$ satisfies global-in-time Strichartz estimates for any admissible pairs.
Besides, it is well known that Strichartz estimates for free Schr\"{o}dinger equation with inverse-square potentials were established
by Burq-Planchon-Stalker-Tahvildar-Zadeh \cite{BPSTZ}. Hence, we have the following global-in-time Strichartz estimate.

\begin{lemma} \cite{BPSTZ, M} \label{Strichartz}
Let $k>0$ and $0<\alpha\leq 2$. Then the solution $u$ of $iu_{t}-H_{\alpha}u=F$ with initial data $u_{0}$ obeys
\begin{align}\label{Strichartzestimates}
\|u\|_{L_{t}^{q}L_{x}^{r}}\lesssim \|u_{0}\|_{L_{x}^{2}}+\|F\|_{L_{t}^{\tilde{q}'}L_{x}^{\tilde{r}'}}
\end{align}
for any $2\leq q, \tilde{q}\leq \infty$ with $\frac{2}{q}+\frac{3}{r}=\frac{2}{\tilde{q}}+\frac{3}{\tilde{r}}=\frac{3}{2}$.
\end{lemma}

Once we have Strichartz estimates Lemma \ref{Strichartz} and Sobolev norm equivalence Lemma \ref{Sobolev}, the local well-posedness
Theorem \ref{localwellposedness} and stability result Lemma \ref{stability} for ($\rm{NLS_{k}}$) can be obtained by the
same proofs as in Theorem 2.15 and Theorem 2.17 of \cite{KMVZ}, respectively.

\begin{lemma} \cite{KMVZ} \label{stability}
For $k>0$ and $0<\alpha\leq 2$. Let $\tilde{u}$ be the solution of
\begin{equation}\label{perturbation}
\left\{ \begin{aligned}
  i&\tilde{u}_{t}-H_{\alpha}\tilde{u}+|\tilde{u}|^{2}\tilde{u}=e,\;\;(t,x) \in {I\times{\bf{R}}^{3}}, \\
  \tilde{u}&(0, x)=\tilde{u}_{0}(x)\in H^{1}({\bf{R}}^{3}),
 \end{aligned}\right.
 \end{equation}
for some 'error' $e$. Given $u_{0}\in H^{1}({\bf{R}}^{3})$ and assume
\begin{align}\label{twoupperbound}
\|u_{0}\|_{H^{1}}+\|\tilde{u}_{0}\|_{H^{1}}\leq A \;\;\text{and}\;\; \|\tilde{u}\|_{L_{t,x}^{5}}\leq M
\end{align}
for some $A$, $M>0$. For any given $\frac{1}{2}\leq s<1$, there exists $\epsilon_{0}=\epsilon_{0}(A, M)$ such that if $0<\epsilon<\epsilon_{0}$ and
\begin{align}\label{error}
\|u_{0}-\tilde{u}_{0}\|_{H^{s}}+\Big\|(1-\Delta)^{\frac{s}{2}}e\Big\|_{N(I)}<\epsilon,
\end{align}
where
$$
N(I):= L_{I,x}^{\frac{10}{7}}+L_{I}^{\frac{5}{3}}L_{x}^{\frac{30}{23}}+L_{I}^{1}L_{x}^{2},
$$
then there exists a solution $u$ of ($\rm{NLS_{k}}$) such that
\begin{align}\label{difference}
\|u-\tilde{u}\|_{S_{\alpha}^{s}(I)}\lesssim_{A, M}\epsilon,
\end{align}
\begin{align}\label{oneupperbound}
\|u\|_{S_{\alpha}^{1}(I)}\lesssim_{A, M}1,
\end{align}
where
$$
S_{\alpha}^{s}(I)=L_{I}^{2}{\mathcal{H}}_{\alpha}^{s,6}\cap L_{I}^{\infty}{\mathcal{H}}_{\alpha}^{s}.
$$
\end{lemma}

Since Mikhlin multiplier theorem for $H_{\alpha}$ holds, naturally, we have the Littlewood-Paley theory associated with
$H_{\alpha}$ (see \cite{KMVZS} for $\alpha=2$). We first give the definition of Littlewood-Paley projection via the heat kernel as
follows: For $N\in 2^{\bf{Z}}$,
\begin{align}\label{projection}
P_{N}:=e^{-\frac{1}{N^{2}}H_{\alpha}}-e^{-\frac{4}{N^{2}}H_{\alpha}}.
\end{align}
We next state Littlewood-Paley decomposition, square function estimate and Bernstein estimates (see \cite{KMVZS} for $\alpha=2$).

\begin{lemma}\label{LPdecompostion}
Let $1<p<\infty$. If $k>0$ and $0<\alpha\leq 2$, then
\begin{align}\label{decompostioneq}
f=\sum_{N\in 2^{{\bf{Z}}}}P_{N}f
\end{align}
as elements of $L^{p}({\bf R}^{3})$. In particular, the sum converges in $L^{p}({\bf R}^{3})$.
\end{lemma}

\begin{lemma}\label{square}
Let $0\leq s<2$ and $1<p<+\infty$. If $k>0$ and $0<\alpha\leq 2$, then
\begin{align}\label{squareestimate}
\Big\|(\sum_{N\in 2^{{\bf{Z}}}}N^{2s}|P_{N}f|^{2})^{\frac{1}{2}}\Big\|_{L^{p}({\bf R}^{3})}\sim \|(H_{\alpha})^{\frac{s}{2}}\|_{L^{p}({\bf R}^{3})}.
\end{align}
\end{lemma}

\begin{lemma} \label{Bernstein}
Let $1<p\leq q<+\infty$ and $s\in {\bf R}$. If $k>0$ and $0<\alpha\leq 2$, then

(i)
\begin{align}\label{Bernsteinpp}
\|P_{N}f\|_{L^{p}({\bf R}^{3})}\lesssim \|f\|_{L^{p}({\bf R}^{3})}.
\end{align}

(ii)
\begin{align}\label{Bernsteinpq}
\|P_{N}f\|_{L^{q}({\bf R}^{3})}\lesssim N^{3(\frac{1}{p}-\frac{1}{q})}\|f\|_{L^{p}({\bf R}^{3})}.
\end{align}

(iii)
\begin{align}\label{Bernsteinsim}
\|(H_{\alpha})^{\frac{s}{2}}P_{N}f\|_{L^{p}({\bf R}^{3})}\sim N^{s}\|f\|_{L^{p}({\bf R}^{3})}.
\end{align}
\end{lemma}

In order to establish linear profile decomposition associated with $e^{-itH_{\alpha}}$ and find a critical element,
we apply the argument of \cite{KMVZE} to get some convergence results. For convenience, define two operators: Let $\{x_{n}\}_{n=1}^{+\infty}\subset {\bf R}^{3}$,
\begin{align}\label{operatortranslation}
H_{\alpha}^{n}:=-\Delta +V(x+x_{n})\;\;\text{and}\;\;
H_{\alpha}^{\infty}:=\left\{
\begin{array}{rcl}
-\Delta+V(x+\bar{x})& &{ x_{n}\rightarrow \bar{x}\in {\bf{R}}^{3}}\\
-\Delta\quad\quad\quad& &{ |x_{n}|\rightarrow +\infty}
\end{array}\right.
\end{align}
Obviously, $\tau_{x_{n}}H_{\alpha}^{n}\psi=H_{\alpha}\tau_{x_{n}}\psi$, where $\tau_{y}\psi(x)=\psi(x-y)$. So $H_{\alpha}$ doesn't commute with
$\tau_{x}$. $H_{\alpha}^{\infty}$ can be regarded as limits of $H_{\alpha}^{n}$ in the following sense (see \cite{KMVZE} for $\alpha=2$).

\begin{lemma}\label{limit}
Let $k>0$ and $1<\alpha< 2$. Assume $t_{n}\rightarrow\bar{t}\in {\bf R}$. and $\{x_{n}\}_{n=1}^{+\infty}\rightarrow\bar{x}\in{\bf R}^{3}$ or $|x_{n}|\rightarrow+\infty$.
Then
\begin{align}\label{operatorlimit}
\lim_{n\rightarrow+\infty}\|(H_{\alpha}^{n}-H_{\alpha}^{\infty})\psi\|_{H^{-1}}=0,\;\;\forall \psi\in H^{1}.
\end{align}
\begin{align}\label{grouplimit}
\lim_{n\rightarrow+\infty}\|(e^{-it_{n}H_{\alpha}^{n}}-e^{-i\bar{t}H_{\alpha}^{\infty}})\psi\|_{H^{-1}}=0,\;\;\forall \psi\in H^{-1}.
\end{align}
\begin{align}\label{fractionaloperatorlimit}
\lim_{n\rightarrow+\infty}\|((1+H_{\alpha}^{n})^{\frac{1}{2}}-(1+H_{\alpha}^{\infty})^{\frac{1}{2}})\psi\|_{L^{2}}=0,\;\;\forall \psi\in H^{1}.
\end{align}
For any $2<q\leq+\infty$ and $\frac{2}{q}+\frac{3}{r}=\frac{3}{2}$.
\begin{align}\label{Strichartzlimit}
\lim_{n\rightarrow+\infty}\|(e^{-itH_{\alpha}^{n}}-e^{-itH_{\alpha}^{\infty}})\psi\|_{L_{t}^{q}L_{x}^{r}}=0,\;\;\forall \psi\in L^{2}.
\end{align}
If $\bar{x}\neq 0$, then for any $t>0$,
\begin{align}\label{heatlimit}
\lim_{n\rightarrow+\infty}\|(e^{-tH_{\alpha}^{n}}-e^{-tH_{\alpha}^{\infty}})\delta_{0}\|_{H^{-1}}=0.
\end{align}
\end{lemma}

\begin{proof}
Here we only give the proof of \eqref{Strichartzlimit} because the others
can be obtained by using the same method of Lemma 3.3 in \cite{KMVZE}, where in the proof of \eqref{operatorlimit}-\eqref{fractionaloperatorlimit} we need to
replace $\dot{H}^{1}$, $\dot{H}^{-1}$, $H_{\alpha}^{n}$, $H_{\alpha}^{\infty}$ and homogeneous Sobolev equivalence Theorem 2.2 in \cite{KMVZE}
with $H^{1}$, $H^{-1}$, $1+H_{\alpha}^{n}$, $1+H_{\alpha}^{\infty}$ and inhomogeneous Sobolev equivalence Lemma \ref{Sobolev}, respectively
and in the proof of \eqref{heatlimit} we need to replace $\dot{H}^{-1}$ by $H^{-1}$.
It only suffices to prove \eqref{Strichartzlimit} in the case $(q, r)=(\infty, 2)$, since the general case can be obtained by
interpolating with the end-point Strichartz estimates (i.e., $(q, r)=(2, 6)$). By density and Strichartz estimates, let $\psi$ be a smooth function with compact support,
so that for $\forall M>0$,
\begin{align}\label{dispersive}
|(e^{-it\Delta}\psi)(x)|\lesssim_{\psi}\langle t\rangle^{-\frac{3}{2}}\Big(1+\frac{|x|}{\langle t\rangle}\Big)^{-M}.
\end{align}
By the definition of $H_{\alpha}^{\infty}$, we consider two cases: $|x_{n}|\rightarrow+\infty$ and $x_{n}\rightarrow\bar{x}$.

For the first case $|x_{n}|\rightarrow+\infty$, we have
\begin{align*}
e^{it\Delta}\psi&=e^{-itH_{\alpha}^{n}}\psi+i\displaystyle\int_{0}^{t}e^{-i(t-s)H_{\alpha}^{n}}V(x+x_{n})e^{is\Delta}\psi ds\\
&=e^{-itH_{\alpha}^{n}}\psi+i\displaystyle\int_{0}^{t}e^{-i(t-s)H_{\alpha}^{n}}(\chi_{|x+x_{n}|\leq R}+\chi_{|x+x_{n}|> R})V(x+x_{n})e^{is\Delta}\psi ds,
\end{align*}
where $R$ is a sufficiently large number that will be chosen later and $\chi_{A}$ is a characteristic function on a set $A$.
Hence by Strichartz estimates \eqref{Strichartzestimates} and dispersive estimates \eqref{dispersive}, we find that
\begin{align}\label{twocases}
\|(e^{-itH_{\alpha}^{n}}-e^{-itH_{\alpha}^{\infty}})\psi\|_{L_{t}^{\infty}L_{x}^{2}}&\lesssim \|\chi_{|x+x_{n}|\leq R}V(x+x_{n})e^{it\Delta}\psi\|_{L_{t,x}^{\frac{10}{7}}}
+\|\chi_{|x+x_{n}|> R}V(x+x_{n})e^{it\Delta}\psi\|_{L_{t}^{q}L_{x}^{r}}\nonumber\\
&\lesssim \Big\|\chi_{|x+x_{n}|\leq R}V(x+x_{n})\langle t\rangle^{-\frac{3}{2}}\Big(1+\frac{|x|}{\langle t\rangle}\Big)^{-M}\Big\|_{L_{t,x}^{\frac{10}{7}}}\nonumber\\
&\quad+\Big\|\chi_{|x+x_{n}|> R}V(x+x_{n})\langle t\rangle^{-\frac{3}{2}}\Big(1+\frac{|x|}{\langle t\rangle}\Big)^{-M}\Big\|_{L_{t}^{q}L_{x}^{r}}\nonumber\\
&:=I_{1}+I_{2},
\end{align}
where $(q,r)=(1,2)$ if $1<\alpha\leq \frac{3}{2}$ and $(q, r)=(\frac{10}{7}, \frac{10}{7})$ if $\frac{3}{2}<\alpha\leq 2$.

We note that $|x_{n}|\geq R$ for $n$ large enough. So for the first part $I_{1}$, $R<|x|\sim |x_{n}|$. If $|x|\leq \langle t\rangle$, then $|x_{n}|\lesssim \langle t\rangle$. Therefore,
\begin{align}\label{nlarge1}
I_{1}\lesssim \|V(x+x_{n})\|_{L_{x}^{\frac{10}{7}}(|x+x_{n}|\leq R)}\Big\||t|^{-\frac{3}{2}}\Big\|_{L_{t}^{\frac{10}{7}}(|t|\gtrsim |x_{n}|)}\lesssim R^{-\alpha+\frac{21}{10}}|x_{n}|^{-\frac{4}{5}}.
\end{align}
If $|x|> \langle t\rangle$,
$$
I_{1}\lesssim \|\chi_{|x+x_{n}|\leq R}V(x+x_{n})\langle t\rangle^{-\frac{3}{2}+M}|x|^{-M}\|_{L_{t,x}^{\frac{10}{7}}}.
$$
When $|t|\leq 1$, it is easy to get
\begin{align}\label{nlarge2}
I_{1}\lesssim  R^{-\alpha+\frac{21}{10}}|x_{n}|^{-M}.
\end{align}
When $|t|>1$ and $M$ sufficiently large,
$$
\|\langle t\rangle^{-\frac{3}{2}+M}\|_{L_{t}^{\frac{10}{7}}(|t|\lesssim |x|)}\lesssim |x|^{-\frac{3}{2}+M+\frac{7}{10}},
$$
so
\begin{align}\label{nlarge3}
I_{1}\lesssim  R^{-\alpha+\frac{21}{10}}|x_{n}|^{-\frac{4}{5}}.
\end{align}
Putting \eqref{nlarge1}, \eqref{nlarge2} and \eqref{nlarge3} together gives that $I_{1}$ tends to 0 as $n$ approaches $+\infty$, regardless of $R>0$.

For the second part $I_{2}$, we consider two subcases: $1<\alpha\leq\frac{3}{2}$ and $\frac{3}{2}<\alpha\leq 2$. If $1<\alpha\leq \frac{3}{2}$, $(q,r)=(1,2)$. When $|x|\leq \langle t\rangle$,
using H\"{o}lder inequality, we have
\begin{align}\label{Rlarge1}
I_{2}\lesssim \|V(x+x_{n})\|_{L_{x}^{\frac{3}{\alpha-}}(|x+x_{n}|> R)}\Big\||t|^{-\frac{3}{2}}\Big\|_{L_{t}^{1}L_{x}^{\frac{6}{3-2(\alpha-)}}(|x|\leq\langle t\rangle)}\lesssim R^{1-\frac{\alpha}{\alpha-}}.
\end{align}
 When $|x|> \langle t\rangle$,
\begin{align}\label{Rlarge2}
I_{2}\lesssim \|V(x+x_{n})\|_{L_{x}^{\frac{3}{\alpha-}}(|x+x_{n}|> R)}\Big\||t|^{-\frac{3}{2}}|x|^{-M}\langle t\rangle^{M}\Big\|_{L_{t}^{1}L_{x}^{\frac{6}{3-2(\alpha-)}}(|x|>\langle t\rangle)}\lesssim R^{1-\frac{\alpha}{\alpha-}}.
\end{align}
It follows from \eqref{Rlarge1} and \eqref{Rlarge2} that $I_{2}$ can be chosen arbitrarily small if we take $R$ sufficiently large.

If $\frac{3}{2}<\alpha\leq 2$, $(q, r)=(\frac{10}{7}, \frac{10}{7})$.
When $|x|\leq \langle t\rangle$,
using H\"{o}lder inequality, we have
\begin{align}\label{Rlarge3}
I_{2}\lesssim \|V(x+x_{n})\|_{L_{x}^{2}(|x+x_{n}|> R)}\Big\||t|^{-\frac{3}{2}}\Big\|_{L_{t}^{\frac{10}{7}}L_{x}^{5}(|x|\leq\langle t\rangle)}\lesssim R^{-\alpha+\frac{3}{2}}.
\end{align}
 When $|x|> \langle t\rangle$,
\begin{align}\label{Rlarge4}
I_{2}\lesssim \|V(x+x_{n})\|_{L_{x}^{2}(|x+x_{n}|> R)}\Big\||t|^{-\frac{3}{2}}|x|^{-M}\langle t\rangle^{M}\Big\|_{L_{t}^{\frac{10}{7}}L_{x}^{5}(|x|>\langle t\rangle)}\lesssim R^{-\alpha+\frac{3}{2}}.
\end{align}
It follows from \eqref{Rlarge3} and \eqref{Rlarge4} that $I_{2}$ can be chosen arbitrarily small if we take $R$ sufficiently large. So we get \eqref{Strichartzlimit} in the case $|x_{n}|\rightarrow+\infty$.

Now we turn to the other case $x_{n}\rightarrow\bar{x}$. Here we obtain that
\begin{align*}
e^{-itH_{\alpha}^{\infty}}\psi&=e^{-itH_{\alpha}^{n}}\psi+i\displaystyle\int_{0}^{t}e^{-i(t-s)H_{\alpha}^{n}}(V(x+x_{n})-V(x+\bar{x}))e^{-isH_{\alpha}^{\infty}}\psi ds
\end{align*}
Hence by Strichartz estimates \eqref{Strichartzestimates}, we have
\begin{align*}
\|(e^{-itH_{\alpha}^{n}}-e^{-itH_{\alpha}^{\infty}})\psi\|_{L_{t}^{\infty}L_{x}^{2}}&\lesssim \|(V(x+x_{n})-V(x+\bar{x}))e^{-itH_{\alpha}^{\infty}}\psi\|_{L_{t}^{2}L_{x}^{\frac{6}{5}}}\nonumber\\
&=\|(V(x+x_{n}-\bar{x})-V(x))e^{-itH_{\alpha}}\tau_{\bar{x}}\psi\|_{L_{t}^{2}L_{x}^{\frac{6}{5}}}
\end{align*}
Replacing $x_{n}-\bar{x}$ by $x_{n}$ and $\tau_{\bar{x}}\psi$ by $\psi$, so we can suppose that $\bar{x}=0$. Hence, for $\forall \epsilon>0$, $|x_{n}|<\epsilon$ when $n$ is large enough. Besides, by
Newton-Leibniz formula,
$$
|V(x+x_{n})-V(x)|\lesssim |x_{n}|\displaystyle\int_{0}^{1}|x+(1-t)x_{n}|^{-\alpha-1}dt.
$$
Using H\"{o}lder inequality and Strichartz estimates \eqref{Strichartzestimates} yields that
\begin{align*}
\|(e^{-itH_{\alpha}^{n}}-e^{-itH_{\alpha}^{\infty}})\psi\|_{L_{t}^{\infty}L_{x}^{2}}
&\lesssim \|(V(x+x_{n})-V(x))\|_{L_{x}^{\frac{15}{11}}(|x|\leq 2\epsilon)}\|e^{-itH_{\alpha}}\psi\|_{L_{t}^{2}L_{x}^{10}}\nonumber\\
&\quad+\|(V(x+x_{n})-V(x))\|_{L_{x}^{\frac{3}{2}}(|x|> 2\epsilon)}\|e^{-itH_{\alpha}}\psi\|_{L_{t}^{2}L_{x}^{6}}\nonumber\\
&\lesssim\epsilon^{\frac{1}{5}}\|\psi\|_{{\mathcal H}_{k}^{\frac{1}{10}}}+|x_{n}|\epsilon^{-(\alpha-1)},
\end{align*}
which implies that
\begin{align*}
\lim_{n\rightarrow+\infty}\|(e^{-itH_{\alpha}^{n}}-e^{-itH_{\alpha}^{\infty}})\psi\|_{L_{t}^{\infty}L_{x}^{2}}\lesssim \epsilon^{\frac{1}{5}}.
\end{align*}
Since $\epsilon$ can be chosen arbitrarily,
\begin{align*}
\lim_{n\rightarrow+\infty}\|(e^{-itH_{\alpha}^{n}}-e^{-itH_{\alpha}^{\infty}})\psi\|_{L_{t}^{\infty}L_{x}^{2}}=0.
\end{align*}
Thus, we get \eqref{Strichartzlimit} in the case $x_{n}\rightarrow\bar{x}$.

\end{proof}

Once getting \eqref{fractionaloperatorlimit} and \eqref{Strichartzlimit}, we follow the same proof of Corollary 3.4 for $\alpha=2$ in \cite{KMVZE} and replace
$\dot{H}^{1}$ and $H_{\alpha}^{n}$
with $H^{1}$ and $1+H_{\alpha}^{n}$ in the procedure of the proof, respectively,
 to get the following decaying estimates.
\begin{lemma}\label{decay}
Let $k>0$, $1<\alpha<2$.
Given $\psi\in H^{1}$, $t_{n}\rightarrow\pm\infty$ and any sequence $\{x_{n}\}_{n=1}^{+\infty}\subset {\bf R}^{3}$. We have
\begin{align}\label{decay6}
\lim_{n\rightarrow+\infty}\|e^{-it_{n}H_{\alpha}^{n}}\psi\|_{L_{x}^{6}}=0.
\end{align}
Moreover, if $\psi\in H^{1}$, then for $2<p<\infty$,
\begin{align}\label{decayp}
\lim_{n\rightarrow+\infty}\|e^{-it_{n}H_{\alpha}^{n}}\psi\|_{L_{x}^{p}}=0.
\end{align}
\end{lemma}

Using \eqref{Strichartzlimit}, Sobolev equivalence Lemma \ref{Sobolev} and interpolation yields the following convergence (see the same result and the detailed proof for $\alpha=2$ in \cite{KMVZ}).
\begin{lemma}\label{Strichartzlimitlemma}
Let $x_{n}\rightarrow\bar{x}\in {\bf R}^{3}$ or $x_{n}\rightarrow\pm\infty$. Then for $\forall\psi\in H^{1}$, $k>0$, $1<\alpha<2$,
\begin{align}\label{Strichartzlimit5}
\lim_{n\rightarrow+\infty}\|(e^{-itH_{\alpha}^{n}}-e^{-itH_{\alpha}^{\infty}})\psi\|_{L_{t,x}^{5}}=0.
\end{align}
\end{lemma}

To get the parameters of linear profile decomposition are asymptotically orthogonal, we finally need two weak convergence results Lemma \ref{weak1} and Lemma \ref{weak2}. Since they
are the direct consequences of Lemma \ref{limit} and the detailed proof can be found in Lemma 3.8 and Lemma 3.9 for $\alpha=2$ in \cite{KMVZE} with a small modification by replacing $\dot{H}^{1}$ and $\Delta$
with $H^{1}$ and $1+\Delta$, respectively, and using the inequality $\|\sqrt{H_{\alpha}^{n}}\psi_{n}\|_{L^{2}}\lesssim \|\psi_{n}\|_{H^{1}}$, where the implicit constant is independent of $n$,
we omit their proof here.
\begin{lemma}\label{weak1}
Let $\psi_{n}\in H^{1}({\bf R}^{3})$ satisfy $\psi_{n}\rightharpoonup 0$ in ${H}^{1}({\bf R}^{3})$ and let $t_{n}\rightarrow\bar{t}\in {\bf R}$.
Then for  $k>0$, $1<\alpha\leq2$,
\begin{align}\label{twiceweak}
e^{-it_{n}H_{\alpha}^{n}}\psi_{n}\rightharpoonup 0 \;\;\text{in}\;\; {H}^{1}({\bf R}^{3}).
\end{align}
\end{lemma}

\begin{lemma}\label{weak2}
Let $\psi\in H^{1}({\bf R}^{3})$ and let $\{(t_{n}, y_{n})\}\subset {\bf R}\times{\bf R}^{3}$, $|t_{n}|\rightarrow\infty$ or $|x_{n}|\rightarrow\infty$ .
Then for  $k>0$, $1<\alpha\leq2$,
\begin{align}\label{onceweak}
(e^{-it_{n}H_{\alpha}^{n}}\psi)(\cdot+y_{n})\rightharpoonup 0 \;\;\text{in}\;\; {H}^{1}({\bf R}^{3}).
\end{align}
\end{lemma}

Use Lemma \ref{limit}, Lemma \ref{decay}, Lemma \ref{weak1}, Lemma \ref{weak2} and Littlewood-Paley theory Lemma \ref{LPdecompostion}- Lemma \ref{Bernstein}
and follow the proof of Proposition 5.1 for $\alpha=2$ in \cite{KMVZ} to give the following linear profile decomposition. Similar to the above, we need to
use ${\mathcal{H}}_{k}^{1}$, $1+H_{\alpha}^{n}$, $1+H_{\alpha}^{\infty}$, and $H^{-1}$ to replace ${\mathcal{\dot{H}}}_{k}^{1}$, $H_{\alpha}^{n}$, $H_{\alpha}^{\infty}$
and $\dot{H}^{-1}$, respectively, in some appropriate places (e.g. in the proof of \eqref{expansion1}).

\begin{lemma}\label{linearprofile}
Let $\{\phi_{n}\}_{n=1}^{+\infty}$ be a uniformly bounded sequence in $H^{1}({\bf R}^{3})$. Then there exist $M^{*}\in {\bf N}\cup\{+\infty\}$, a subsequence of $\{\phi_{n}\}_{n=1}^{M^{*}}$,
which is denoted by itself, such that for $k>0$, $1<\alpha\leq2$,  the following statements hold.

(1) For each $1\leq j\leq M\leq M^{*}$, there exist (fixed in $n$) a profile $\psi^{j}$ in $H^{1}({\bf R}^{3})$, a sequence (in $n$) of time shifts $t_{n}^{j}$ and a sequence (in $n$) of space
shifts $x_{n}^{j}$, and  there exists a sequence (in $n$) of remainder $W_{n}^{M}$ in $H^{1}({\bf R}^{3})$ such that
\begin{align}\label{decomposition}
\phi_{n}=\sum_{j=1}^{M}e^{it_{n}^{j}H_{\alpha}}\tau_{x_{n}^{j}}\psi^{j}+W_{n}^{M}:=\sum_{j=1}^{M}\psi_{n}^{j}+W_{n}^{M}.
\end{align}

(2) For each $1\leq j\leq M$,
\begin{align}\label{zeroinfty}
\text{either}\; t_{n}^{j}=0\; \text{for any}\; n\in {\bf N}\;\;\; \text{or}\; \lim_{n\rightarrow+\infty}t_{n}^{j}=\pm\infty\\
\text{either}\; x_{n}^{j}=0\; \text{for any}\; n\in {\bf N}\;\;\; \text{or}\; \lim_{n\rightarrow+\infty}|x_{n}^{j}|=+\infty.
\end{align}

(3) The time and space sequence have a pairwise divergence property. Namely, for $1\leq j\neq k\leq M$,
\begin{align}\label{divergence}
\lim_{n\rightarrow+\infty}(|t_{n}^{j}-t_{n}^{k}|+|x_{n}^{j}-x_{n}^{k}|)=+\infty.
\end{align}

(4) The remainder sequence has the following asymptotic smallness property and weak convergence property:
\begin{align}\label{smallness}
 \lim_{M\rightarrow M^{*}}\overline{\lim}_{n\rightarrow+\infty}\|e^{-itH_{\alpha}}W_{n}^{M}\|_{L_{t,x}^{5}}=0
\end{align}
and
\begin{align}\label{remainderweak}
\tau_{-x_{n}^{M}}e^{-it_{n}^{M}H_{\alpha}}W_{n}^{M}\rightharpoonup 0 \;\;\text{in}\;\; H^{1},\;\;\text{as}\;\; n\rightarrow+\infty.
\end{align}

(5) For each fixed $M$, we have the asymptotic Pythagorean expansion as follows:
\begin{align}\label{expansion0}
\|\phi_{n}\|_{L^{2}}^{2}=\sum_{j=1}^{M}\|\psi^{j}\|_{L^{2}}^{2}+\|W_{n}^{M}\|_{L^{2}}^{2}+o_{n}(1),
\end{align}
\begin{align}\label{expansion1}
\|\phi_{n}\|_{{\mathcal{\dot{H}}}_{k}^{1}}^{2}=\sum_{j=1}^{M}\|\tau_{x_{n}^{j}}\psi^{j}\|_{{\mathcal{\dot{H}}}_{k}^{1}}^{2}+\|W_{n}^{M}\|_{{\mathcal{\dot{H}}}_{k}^{1}}^{2}+o_{n}(1)
\end{align}
and
\begin{align}\label{expansion4}
\|\phi_{n}\|_{L^{4}}^{4}=\sum_{j=1}^{M}\|\psi_{n}^{j}\|_{L^{4}}^{4}+\|W_{n}^{M}\|_{L^{4}}^{4}+o_{n}(1),
\end{align}
where $o_{n}(1)\rightarrow 0$ as $n\rightarrow+\infty$. Specially,
\begin{align}\label{expansions}
S_{k}(\phi_{n})=\sum_{j=1}^{M}S_{k}(\psi_{n}^{j})+S_{k}(W_{n}^{M})+o_{n}(1)
\end{align}
and
\begin{align}\label{expansioni}
I_{k}(\phi_{n})=\sum_{j=1}^{M}I_{k}(\psi_{n}^{j})+I_{k}(W_{n}^{M})+o_{n}(1).
\end{align}
\end{lemma}

Following the proof of Theorem 6.1 for $\alpha=2$  in \cite{KMVZ} and using Theorem \ref{scattering0}, Lemma \ref{stability}, Lemma \ref{limit} and Lemma \ref{Strichartzlimitlemma},
replacing $\frac{1}{|x|^{2}}$ and homogenous Sobolev spaces with $\frac{1}{|x|^{\alpha}}$ inhomogeneous Sobolev spaces, respectively,
in some appropriate places gives the following lemma on nonlinear
profiles when $|x_{n}|\rightarrow+\infty$.

\begin{lemma}\label{nonlinearprofile}
Let $k>0$, $1<\alpha\leq2$, $\psi\in H^{1}({\bf R}^{3})$ and the time sequence $t_{n}\equiv 0$ or $t_{n}\rightarrow\pm\infty$ such that
\begin{align}\label{time0}
\text{if}\;\; t_{n}\equiv 0,\;\;S_{0}(\psi)<n_{0}\;\;\text{ and}\;\; P_{0}(\psi)\geq 0
\end{align}
and
\begin{align}\label{timeinfty}
\text{if}\;\; t_{n}\rightarrow\pm\infty,\;\;\frac{1}{2}\|\psi\|_{H^{1}}^{2}<n_{0}.
\end{align}
Let
$$
\psi_{n}=e^{-it_{n}H_{\alpha}}\tau_{x_{n}}\psi,
$$
where the space sequence $x_{n}$ satisfies $|x_{n}|\rightarrow+\infty$. Then taking $n$ large enough, we have that
the solution $u(t):=NLS_{k}(t)\psi_{n}$ of ($\rm{NLS_{k}}$) with initial data $u_{0}=\psi_{n}$ is global and satisfies
$$
\|NLS_{k}(t)\psi_{n}\|_{S_{\alpha}^{1}(I)}\lesssim_{\|\psi\|_{H^{1}}}1.
$$
Moreover, for $\forall \epsilon >0$, there exist a positive number $N=N(\epsilon)$ and a smooth compact supported function $\chi_{\epsilon}$ on ${\bf R}\times {\bf R}^{3}$ satisfying
\begin{align}\label{dense}
\|NLS_{k}(t)\psi_{n}(x)-\chi_{\epsilon}(t+t_{n}, x-x_{n})\|_{Z}<\epsilon\;\;\text{for}\;\;n\geq N,
\end{align}
where the norm
$$
\|f\|_{Z}:=\|f\|_{L_{t,x}^{5}}+\|f\|_{L_{t,x}^{\frac{10}{3}}}+\|f\|_{L_{t}^{\frac{30}{7}}L_{x}^{\frac{90}{31}}}+\|f||_{L_{t}^{\frac{30}{7}}{\mathcal{H}}_{k}^{\frac{31}{60},\frac{90}{31}}}
$$
\end{lemma}

\section{Variational characterization}\label{section3}
\setcounter{equation}{0}
We start with proving the positivity of $K_{k}^{a, b}$ near $0$ in the $H^{1}({\bf R}^{3})$.
\begin{lemma}\label{positive}
If the uniform $L^{2}$-bounded sequence $\varphi_{n}\in H^{1}({\bf R}^{3})\setminus \{0\}$ satisfies $\lim_{n\rightarrow+\infty}\|\nabla \varphi_{n}\|_{L^{2}}=0$, then
for sufficiently large $n\in {\bf N}$, $K_{k}^{a, b}(\varphi_{n})>0$.
\end{lemma}

\begin{proof}
By the fact that $-2V\leq x\cdot\nabla V\leq 0$ and $V\geq 0$, we always have that for large enough $n$
\begin{align*}
K_{k}^{a, b}(\varphi_{n})
&=\frac{2a+b}{2}\displaystyle\int_{{\bf R}^{3}}|\nabla \varphi_{n}(x)|^{2}dx
+\frac{2a+3b}{2}\displaystyle\int_{{\bf R}^{3}}V|\varphi_{n}(x)|^{2}dx
+\frac{b}{2}\displaystyle\int_{{\bf R}^{3}}x\cdot\nabla V |\varphi_{n}(x)|^{2}dx\\
&\;\;+\frac{2a+3b}{2}\displaystyle\int_{{\bf R}^{3}}|\varphi_{n}(x)|^{2}dx
-\frac{4a+3b}{4}\displaystyle\int_{{\bf R}^{3}}|\varphi_{n}(x)|^{4}dx\\
&\geq \frac{2a+b}{2}\displaystyle\int_{{\bf R}^{3}}|\nabla \varphi_{n}(x)|^{2}dx
+\frac{2a+3b}{2}\displaystyle\int_{{\bf R}^{3}}|\varphi_{n}(x)|^{2}dx
-\frac{4a+3b}{4}\displaystyle\int_{{\bf R}^{3}}|\varphi_{n}(x)|^{4}dx\\
&\geq \frac{2a+b}{2}\displaystyle\int_{{\bf R}^{3}}|\nabla \varphi_{n}(x)|^{2}dx
-\frac{4a+3b}{4}\displaystyle\int_{{\bf R}^{3}}|\varphi_{n}(x)|^{4}dx\\
&\geq \frac{2a+b}{2}\displaystyle\int_{{\bf R}^{3}}|\nabla \varphi_{n}(x)|^{2}dx
-\frac{4a+3b}{4}c\Big(\displaystyle\int_{{\bf R}^{3}}|\varphi_{n}(x)|^{2}dx\Big)\Big(\displaystyle\int_{{\bf R}^{3}}|\nabla\varphi_{n}(x)|^{2}dx\Big)^{\frac{3}{2}}\\
&\geq \frac{2a+b}{4}\displaystyle\int_{{\bf R}^{3}}|\nabla \varphi_{n}(x)|^{2}dx>0,
\end{align*}
where we have used the Gagliardo-Nirenberg inequality in the line before last and the assumptions $\lim_{n\rightarrow+\infty}\|\nabla \varphi_{n}\|_{L^{2}}=0$
and uniform boundedness of $\|\varphi_{n}\|_{L^{2}}$ in the last line.
\end{proof}

Simple computation gives
$$
\|\nabla \varphi_{\lambda}^{a, b}\|_{L^{2}}=e^{\frac{1}{2}(2a+b)\lambda}\|\nabla \varphi\|_{L^{2}},
$$
which implies that
\begin{align}\label{gradient0}
\lim_{\lambda\rightarrow-\infty}\|\nabla \varphi_{\lambda}^{a, b}\|_{L^{2}}=0.
\end{align}
So it follows from Lemma \ref{positive} that for sufficiently small $\lambda<0$,
\begin{align}\label{positiveK}
K_{k}^{a, b}(\varphi_{\lambda}^{a, b})>0.
\end{align}

For brevity, let $\overline{\mu}=2a+b$ and $\underline{\mu}=2a+3b$. Next we introduce the functional
$$
J_{k}^{a,b}(\varphi):=\frac{1}{\overline{\mu}}(\overline{\mu}-{\mathcal{L}}^{a,b})S_{k}(\varphi)
=\frac{1}{\overline{\mu}}(\overline{\mu}S_{k}(\varphi)-K_{k}^{a, b}(\varphi)).
$$
The lemma shows that the positivity of $J_{k}^{a,b}(\varphi)$ and the monotonicity of $J_{k}^{a,b}(\varphi_{\lambda}^{a, b})$ in $\lambda$.

\begin{lemma}\label{monotonicity}
For any $\varphi\in H^{1}({\bf R}^{3})$, we have the following two identities:
\begin{align}\label{positiveJ}
\overline{\mu}J_{k}^{a,b}(\varphi)&=(\overline{\mu}-{\mathcal{L}}^{a,b})S_{k}^{a,b}(\varphi)\nonumber\\
&=-b\displaystyle\int_{{\bf R}^{3}}V|\varphi|^{2}dx-\frac{b}{2}\displaystyle\int_{{\bf R}^{3}}(x\cdot\nabla V)|\varphi|^{2}dx-b\|\varphi\|_{L^{2}}^{2}+\frac{a+b}{2}\|\varphi\|_{L^{4}}^{4}.
\end{align}
and
\begin{align}\label{monotonicityJ}
({\mathcal{L}}^{a,b}-\overline{\mu})(\underline{\mu}-{\mathcal{L}}^{a,b})S_{k}^{a,b}(\varphi)=b^{2}\displaystyle\int_{{\bf R}^{3}}(-3x\cdot\nabla V-x\nabla^{2} Vx^{T})|\varphi|^{2}dx,
+(a+b)a\|\varphi\|_{L^{4}}^{4}
\end{align}
where $\nabla^{2} V$ is Hessian matric of $V$.
\end{lemma}

\begin{proof}
By simple computations, we have
$$
{\mathcal{L}}^{a,b}\|\nabla\varphi\|_{L^{2}}^{2}=\overline{\mu}\|\nabla\varphi\|_{L^{2}}^{2},\;\;
{\mathcal{L}}^{a,b}\|\varphi\|_{L^{2}}^{2}=\underline{\mu}\|\varphi\|_{L^{2}}^{2},\;\;
{\mathcal{L}}^{a,b}\|\varphi\|_{L^{4}}^{4}=(4a+3b)\|\varphi\|_{L^{4}}^{4},
$$
$$
{\mathcal{L}}^{a,b}\Big(\displaystyle\int_{{\bf R}^{3}}V|\varphi|^{2}dx\Big)=\underline{\mu}\displaystyle\int_{{\bf R}^{3}}V|\varphi|^{2}dx+b\displaystyle\int_{{\bf R}^{3}}(x\cdot\nabla V)|\varphi|^{2}dx,
$$
and
$$
{\mathcal{L}}^{a,b}\Big(\displaystyle\int_{{\bf R}^{3}}(x\cdot\nabla V)|\varphi|^{2}dx\Big)=2(a+b)\displaystyle\int_{{\bf R}^{3}}x\cdot\nabla V|\varphi|^{2}dx
+b\displaystyle\int_{{\bf R}^{3}}x\nabla^{2} Vx^{T}|\varphi|^{2}dx,
$$
which imply that \eqref{positiveJ} and \eqref{monotonicityJ}. We conclude the proof.
\end{proof}

\begin{remark}\label{positive-monotonicity}
By the definition of $V$, we have $2V+x\cdot V\geq 0$ and $3x\cdot\nabla V+x\nabla^{2}Vx^{T}\leq 0$,
which together with \eqref{positiveJ} and \eqref{monotonicityJ} yield that $J_{k}^{a,b}(\varphi)> 0$ and
${\mathcal{L}}^{a,b}J_{k}^{a,b}(\varphi)\geq 0$ for any $\varphi\in H^{1}({\bf R}^{3})\setminus\{0\}$ which implies that
the function $\lambda \mapsto J_{k}^{a,b}(\varphi_{\lambda}^{a, b})$ is increasing.
\end{remark}

Using the functional $K_{k}^{a,b}$ with $(a, b)$ satisfying \eqref{parameter} , we introduce a general minimizing problem:
\begin{align}\label{thresholdab}
n_{k}^{a,b}=\inf\{S_{k}(\varphi):\varphi\in H^{1}({\bf R}^{3})\setminus \{0\}, K_{k}^{a,b}(\varphi)=0\}.
\end{align}
In particular, when $(a,b)=(3,-2)$, it is namely $n_{k}$ in \eqref{threshold}. In fact, we shall show that
\begin{align}\label{threshold4}
n_{k}^{a,b}=n_{0}^{a,b}=n_{k}=n_{0}.
\end{align}
To this end, we introduce another one general minimizing problem in terms of the positive functional $J_{0}^{a,b}$ for any ${a, b}$ satisfying \eqref{parameter}:
\begin{align}\label{thresholdJ}
j^{a,b}=\inf\{J_{0}^{a,b}(\varphi):\varphi\in H^{1}({\bf R}^{3})\setminus \{0\}, K_{0}^{a,b}(\varphi)\leq 0\}.
\end{align}
In the following lemma, we shall show a relation between the two minimizers $n_{0}^{a, b}$ in \eqref{thresholdab} (i.e., $k=0$)  and $j^{a,b}$ in \eqref{thresholdJ}.

\begin{lemma}\label{j=n0}
\begin{align}\label{jab=n0}
j^{a,b}=n_{0}^{a, b}.
\end{align}
\end{lemma}

\begin{proof}
On one hand, for any $\varphi\in H^{1}({\bf R}^{3})\setminus \{0\}$ with $K_{0}^{a,b}(\varphi)\leq 0$, there are two possibilities:
$K_{0}^{a,b}(\varphi)= 0$ and $K_{0}^{a,b}(\varphi)< 0$. When $K_{0}^{a,b}(\varphi)=0$, $S_{0}(\varphi)=J_{0}^{a,b}(\varphi)$ by the definition of $J_{k}^{a, b}$.
Hence,
\begin{align}\label{onehand}
n_{0}^{a, b}\leq J_{0}^{a,b}(\varphi).
\end{align}
When $K_{0}^{a,b}(\varphi)=K_{0}^{a,b}(\varphi_{0}^{a, b})< 0$, using the continuity of $K_{0}^{a, b}(\varphi_{\lambda}^{a, b})$ in $\lambda$ and the fact that
for sufficiently small $\lambda<0$ such that $K_{0}^{a, b}(\varphi_{\lambda}^{a, b})>0$ holds by \eqref{positiveK} yields that
there exists a $\lambda_{0}<0$ such that $K_{0}^{a, b}(\varphi_{\lambda_{0}}^{a, b})=0$.
Using Remark \ref{positive-monotonicity}, we get
$$
S_{0}(\varphi_{\lambda_{0}}^{a, b})=J_{0}^{a, b}(\varphi_{\lambda_{0}}^{a, b})\leq J_{0}^{a,b}(\varphi_{0}^{a, b})=J_{0}^{a,b}(\varphi)
$$
Hence, we still have \eqref{onehand}.
Altogether, for any $\varphi\in H^{1}({\bf R}^{3})\setminus \{0\}$ with $K_{0}^{a,b}(\varphi)\leq 0$, we have \eqref{onehand}. By the definition of $j^{a, b}$,
we have $n_{0}^{a,b}\leq j^{a,b}$.

On the other hand, for any $\varphi\in H^{1}({\bf R}^{3})\setminus \{0\}$ with $K_{0}^{a,b}(\varphi)=0$. Of course, $K_{0}^{a,b}(\varphi)\leq 0$
and $S_{0}(\varphi)=J_{0}^{a,b}(\varphi)$, which implies that $j^{a, b}\leq n_{0}^{a,b}$. Thus, we complete the proof.
\end{proof}

Next we shall apply Lemma \ref{j=n0} to get $n_{k}^{a,b}=n_{0}^{a,b}$.

\begin{lemma}\label{n=n0}
\begin{align}\label{nab=n0}
n_{k}^{a,b}=n_{0}^{a, b}.
\end{align}
\end{lemma}

\begin{proof}
On one hand, for any $\varphi\in H^{1}({\bf R}^{3})\setminus \{0\}$ with $K_{k}^{a,b}(\varphi)= 0$, we have
$S(\varphi)=J^{a,b}(\varphi)$ by the definition of $J_{k}^{a, b}$. Since $V\geq 0$ and $2V+x\cdot\nabla V\geq 0$, we get
$K_{0}^{a,b}(\varphi)\leq K_{k}^{a,b}(\varphi)=0$ and $J_{0}^{a,b}(\varphi)\leq J_{k}^{a,b}(\varphi)$, which together with \eqref{jab=n0} implies that
$$
n_{0}^{a, b}=j^{a,b}\leq J_{k}^{a,b}(\varphi).
$$
Also as $K^{a,b}(\varphi)= 0$, taking infimum on both sides of the above inequality yields that
$$
n_{0}^{a, b}\leq n_{k}^{a, b}.
$$

On the other hand, we review that $Q$ is the positive radial exponential decaying solution of the elliptic equation \eqref{ellipticequation}, so there exists a constant $c$ such
that $Q(x)\lesssim e^{-c|x|}$ for any $x\in {\bf R}^{3}$ (e.g. See Theorem 8.1.1 in \cite{C}. Here we only need the decaying property of $Q$, not necessarily exponential decaying).
Let $x_{n}$ be a sequence satisfying $|x_{n}|\rightarrow+\infty$. Hence, for any given $R>0$, there exists $N=N(R)>0$, for any
$n\geq N$, we have $|x_{n}|\geq 2R$. we claim that
\begin{align}\label{V01}
\displaystyle\int_{{\bf R}^{3}}V(x)|Q(x-x_{n})|^{2}dx\rightarrow 0\;\;\text{as}\;\; n\rightarrow +\infty.
\end{align}
Indeed,
\begin{align*}
\displaystyle\int_{{\bf R}^{3}}V(x)|Q(x-x_{n})|^{2}dx
&\leq \displaystyle\int_{|x|\leq R}V(x)|Q(x-x_{n})|^{2}dx
+\displaystyle\int_{|x|> R}V(x)|Q(x-x_{n})|^{2}dx\\
&:=I_{1}+I_{2}.
\end{align*}
For the first part $I_{1}$, when $n\geq N$, $|x-x_{n}|\geq \frac{|x_{n}|}{2}$ and then
$$
|Q(x-x_{n})|\leq \sup_{|x|\geq\frac{|x_{n}|}{2}}|Q(x)|:=C(|x_{n}|)\rightarrow 0
$$
as $n$ tends to $+\infty$. Therefore,
\begin{align*}
I_{1}\lesssim C(|x_{n}|)\displaystyle\int_{|x|\leq R}V(x)dx\lesssim R^{(3-\alpha)}C(|x_{n}|).
\end{align*}
For the second part $I_{2}$, $V(x)\lesssim R^{-\alpha}$, which implies that
$$
I_{2}\lesssim R^{-\alpha}\displaystyle\int_{{\bf R}^{3}}|Q(x)|^{2}dx.
$$
Combining the above two parts yields that the claim \eqref{V01} holds true. Noticing that $x\cdot\nabla V=(-\alpha) V$, thus we also have
\begin{align}\label{nablaV0}
\displaystyle\int_{{\bf R}^{3}}(x\cdot\nabla V)|Q(x-x_{n})|^{2}dx\rightarrow 0\;\;\text{as}\;\; n\rightarrow +\infty.
\end{align}

Using the definition of $K_{k}^{a, b}$ \eqref{functionalK} and $K_{0}^{a, b}(Q)=0$ and taking $\theta$ sufficiently large, we have
$$
K_{k}^{a, b}(\tau_{x_{n}}Q)>K_{0}^{a, b}(\tau_{x_{n}}Q)=K_{0}^{a, b}(Q)=0
$$
and
$$
K_{k}^{a, b}(\theta\tau_{x_{n}}Q)<0,
$$
from which it follows that there must be $\theta_{n}>1$ satisfying $K_{k}^{a, b}(\theta_{n}\tau_{x_{n}}Q)=0$. We claim that
\begin{align}\label{theta1}
\lim_{n\rightarrow +\infty}\theta_{n}=1.
\end{align}
In fact, by $K_{0}^{a, b}(Q)=0$, we have
$$
\frac{2a+b}{2}\displaystyle\int_{{\bf R}^{3}}|\nabla Q(x)|^{2}dx
+\frac{2a+3b}{2}\displaystyle\int_{{\bf R}^{3}}|Q(x)|^{2}dx
=\frac{4a+3b}{4}\displaystyle\int_{{\bf R}^{3}}|Q(x)|^{4}dx,
$$
which implies that
\begin{align*}
K_{k}^{a, b}(\theta_{n}\tau_{x_{n}}Q)&=\theta_{n}^{2}\Big[\frac{2a+3b}{2}\displaystyle\int_{{\bf R}^{3}}V|Q(x-x_{n})|^{2}dx
+\frac{b}{2}\displaystyle\int_{{\bf R}^{3}}(x\cdot\nabla V) |Q(x-x_{n})|^{2}dx\\
&\;\;+\frac{4a+3b}{4}(1-\theta_{n}^{2})\displaystyle\int_{{\bf R}^{3}}|Q(x)|^{4}dx\Big]=0.
\end{align*}
Hence,
$$
\frac{2a+3b}{2}\displaystyle\int_{{\bf R}^{3}}V|Q(x-x_{n})|^{2}dx
+\frac{b}{2}\displaystyle\int_{{\bf R}^{3}}(x\cdot\nabla V) |Q(x-x_{n})|^{2}dx
+\frac{4a+3b}{4}(1-\theta_{n}^{2})\displaystyle\int_{{\bf R}^{3}}|Q(x)|^{4}dx=0.
$$
Taking limit in the above quality and using \eqref{V01} and \eqref{nablaV0} gives the claim \eqref{theta1}.

Using \eqref{V01} and \eqref{theta1} yields that
$$
\lim_{n\rightarrow+\infty}S_{k}(\theta_{n}\tau_{x_{n}}Q)=S_{0}(Q)=n_{0}^{a,b},
$$
which together with $K_{k}^{a, b}(\theta_{n}\tau_{x_{n}}Q)=0$ implies that
$$
n_{k}^{a, b}\leq n_{0}^{a,b}.
$$
We conclude the proof.
\end{proof}

Lemma \ref{n=n0} shows that $n_{k}^{a,b}=n_{0}^{a,b}=S_{0}(Q)$, so $n_{k}^{a,b}$ and $n_{0}^{a,b}$ don't depend on the parameters $a$ and $b$.
Therefore, \eqref{threshold4} holds true.
It is known that $n_{0}$ is attained by $Q$. However, the following lemma shows that $n_{k}$ is never attained for any $k>0$.

\begin{lemma}\label{attain}
$n_{k}$ is never attained for any $k>0$.
\end{lemma}

\begin{proof}
Suppose that there exists a $\varphi\in H^{1}({\bf R}^{3})\setminus \{0\}$ such that $n_{k}$ is attained by $\varphi$.
Namely, $P_{k}(\varphi)=0$ and $S_{k}(\varphi)=n_{k}$.
As $\varphi\in H^{1}({\bf R}^{3})\setminus \{0\}$, $\lim_{|x|\rightarrow +\infty}\varphi(x)=0$. Following the proof of \eqref{V01}, we have
\begin{align}\label{V0}
\displaystyle\int_{{\bf R}^{3}}V(x)|\varphi(x-x_{n})|^{2}dx\rightarrow 0\;\;\text{as}\;\; n\rightarrow +\infty.
\end{align}
where $|x_{n}|\rightarrow+\infty$ as $n\rightarrow +\infty$.
Hence, we have
\begin{align*}
-\displaystyle\int_{{\bf R}^{3}}(x\cdot\nabla V) |\tau_{x_{n}}\varphi|^{2}dx\; \text{ is positive and tends to zero as}\;\; n\rightarrow +\infty
\end{align*}
and
\begin{align*}
2\displaystyle\int_{{\bf R}^{3}}V|\tau_{x_{n}}\varphi|^{2}dx
+\displaystyle\int_{{\bf R}^{3}}(x\cdot\nabla V )|\tau_{x_{n}}\varphi|^{2}dx\; \text{ is positive and tends to zero as}\;\; n\rightarrow +\infty.
\end{align*}
Therefore, for $n$ sufficiently large, we have
\begin{align}\label{K0}
P_{k}(\tau_{x_{n}}\varphi)<P_{k}(\varphi)=0
\end{align}
and
\begin{align}\label{J}
J_{k}^{3, -2}(\tau_{x_{n}}\varphi)<J_{k}^{3,-2}(\varphi)=S_{k}(\varphi)=n_{k}.
\end{align}
By \eqref{positiveK}, for sufficiently small $\lambda<0$, we have $P_{k}((\tau_{x_{n}}\varphi)_{\lambda}^{3, -2})>0$, which combined with \eqref{K0} implies that
there exists a $\lambda_{0}<0$ such that $P_{k}((\tau_{x_{n}}\varphi)_{\lambda_{0}}^{3, -2})=0$.
Using Remark \ref{positive-monotonicity} and \eqref{J}, we get
$$
n_{k}=S_{k}((\tau_{x_{n}}\varphi)_{\lambda_{0}}^{3, -2})=J_{k}^{3, -2}((\tau_{x_{n}}\varphi)_{\lambda_{0}}^{3, -2})
< J_{k}^{3,-2}((\tau_{x_{n}}\varphi)_{0}^{3, -2})=J_{k}^{3,-2}(\tau_{x_{n}}\varphi)<n_{k},
$$
which is impossible. Thus we complete the proof.
\end{proof}

To get the fact that $P_{k}(\varphi)$ in \eqref{functionalP} and $I_{k}(\varphi)$ in \eqref{functionalI}
have the same sign under the condition $S_{k}(\varphi)<n_{0}$ with $\varphi\in H^{1}({\bf R}^{3})$, we introduce
${\mathcal{N}}_{a,b}^{\pm}\subset H^{1}({\bf R}^{3})$ defined by
\begin{align}\label{nab+}
{\mathcal{N}}_{a,b}^{+}:=\{\varphi\in H^{1}({\bf R}^{3}): S_{k}(\varphi)<n_{0}, K_{k}^{a,b}(\varphi)\geq 0\}
\end{align}
and
\begin{align}\label{nab-}
{\mathcal{N}}_{a,b}^{-}:=\{\varphi\in H^{1}({\bf R}^{3}): S_{k}(\varphi)<n_{0}, K_{k}^{a,b}(\varphi)< 0\}.
\end{align}
It is easy to see that ${\mathcal{N}}_{3,-2}^{\pm}={\mathcal{N}}^{\pm}$ in \eqref{n+} and \eqref{n-}.
The above fact can be obtained if
we show that ${\mathcal{N}}_{a,b}^{\pm}$ are independent of $(a, b)$
(i.e., ${\mathcal{N}}_{a,b}^{\pm}={\mathcal{N}}^{\pm}$), which follows from the contractivity
of ${\mathcal{N}}_{a,b}^{+}$.

\begin{lemma}\label{independentofab}
Let $(a, b)$ satisfy \eqref{parameter}, then ${\mathcal{N}}_{a,b}^{\pm}$ are independent of $(a, b)$.
\end{lemma}

\begin{proof}
The proof is similar to the one of Lemma 2.9 in \cite{IMN} (see also Lemma 2.15 in \cite{II}), so we omit it.
\end{proof}

The following lemma shows that for any element $\varphi$ in ${\mathcal{N}}^{+}$,
$S_{k}(\varphi)\sim \|\varphi\|_{{\mathcal{H}}_{k}^{1}}\sim \|\varphi\|_{H^{1}}$.

\begin{lemma}\label{equivalentSH}
Let $\varphi\in {\mathcal{N}}^{+}$, then
\begin{align}\label{equivalenceSH}
\frac{1}{4}\|\varphi\|_{{\mathcal{H}}_{k}^{1}}^{2}\leq S_{k}(\varphi)\leq\frac{1}{2}\|\varphi\|_{{\mathcal{H}}_{k}^{1}}^{2}.
\end{align}
\end{lemma}

\begin{proof}
By Lemma \ref{independentofab}, we have $I_{k}(\varphi)\geq 0$, which implies that
$$
\displaystyle\int_{{\bf R}^{3}}|\varphi(x)|^{4}dx\leq \|\varphi\|_{{\mathcal{H}}_{k}^{1}}^{2}.
$$
Thus, we have
$$
\frac{1}{2}\|\varphi\|_{{\mathcal{H}}_{k}^{1}}^{2}\geq
S_{k}(\varphi)=\frac{1}{2}\|\varphi\|_{{\mathcal H}_{k}^{1}}-\frac{1}{4}\displaystyle\int_{{\bf R}^{3}}|\varphi(x)|^{4}dx
\geq \frac{1}{4}\|\varphi\|_{{\mathcal{H}}_{k}^{1}}^{2},
$$
which is namely \eqref{equivalenceSH}. We complete the proof.
\end{proof}

Finally, we obtain the corresponding uniform bounds on $P_{k}(\varphi)$ when $\varphi\in {\mathcal{N}}^{\pm}$, which plays a vital role in
the proof of Theorem \ref{scattering1}.

\begin{lemma}\label{uniformbounds}

1. Let $\varphi\in {\mathcal{N}}^{-}$, then
\begin{align}\label{upperbound}
P_{k}(\varphi)\leq -4\Big(n_{0}-S_{k}(\varphi)\Big)
\end{align}

2. Let $\varphi\in {\mathcal{N}}^{+}$, then
\begin{align}\label{lowerbound}
P_{k}(\varphi)\geq \min\Big\{4\Big(n_{0}-S_{k}(\varphi)\Big),
\frac{2}{5}\Big(\|\nabla\varphi\|_{L^{2}}^{2}-\frac{1}{2}\displaystyle\int_{{\bf R}^{3}}(x\cdot\nabla V) |\varphi(x)|^{2}dx)\Big)\Big\}.
\end{align}
\end{lemma}

\begin{proof}
For any $\varphi\in H^{1}({\bf R}^{3})$, define $s(\lambda):=S_{k}(\varphi_{\lambda}^{3,-2})$, then
\begin{align}
&s(\lambda)=\frac{1}{2}e^{4\lambda}\displaystyle\int_{{\bf R}^{3}}|\nabla \varphi(x)|^{2}dx
+\frac{1}{2}\displaystyle\int_{{\bf R}^{3}}V(e^{-2\lambda}x)|\varphi(x)|^{2}dx\nonumber\\
&\quad\quad\quad +\frac{1}{2}\displaystyle\int_{{\bf R}^{3}}|\varphi(x)|^{2}dx
-\frac{1}{4}e^{6\lambda}\displaystyle\int_{{\bf R}^{3}}|\varphi(x)|^{4}dx,\nonumber\\
&s'(\lambda)=P_{k}(\varphi_{\lambda}^{3,-2})=2e^{4\lambda}\displaystyle\int_{{\bf R}^{3}}|\nabla \varphi(x)|^{2}dx
-e^{-2\lambda}\displaystyle\int_{{\bf R}^{3}}\Big[x\cdot (\nabla V)(e^{-2\lambda}x)\Big]|\varphi(x)|^{2}dx\label{s'}\\
&\quad\quad\quad-\frac{3}{2}e^{6\lambda}\displaystyle\int_{{\bf R}^{3}}|\varphi(x)|^{4}dx\nonumber\\
&\quad\quad\geq 2e^{4\lambda}\displaystyle\int_{{\bf R}^{3}}|\nabla \varphi(x)|^{2}dx
-\frac{3}{2}e^{6\lambda}\displaystyle\int_{{\bf R}^{3}}|\varphi(x)|^{4}dx,\nonumber\\
&s''(\lambda)=8e^{4\lambda}\displaystyle\int_{{\bf R}^{3}}|\nabla \varphi(x)|^{2}dx
+2e^{-2\lambda}\displaystyle\int_{{\bf R}^{3}}\Big[x\cdot (\nabla V)(e^{-2\lambda}x)\Big]|\varphi(x)|^{2}dx \label{s''}\\
&\quad\quad\quad+2e^{-4\lambda}\displaystyle\int_{{\bf R}^{3}}\Big[x(\nabla^{2} V)(e^{-2\lambda}x)x^{T}\Big]|\varphi(x)|^{2}dx
-9e^{6\lambda}\displaystyle\int_{{\bf R}^{3}}|\varphi(x)|^{4}dx\nonumber\\
&\quad\quad\leq 8e^{4\lambda}\displaystyle\int_{{\bf R}^{3}}|\nabla \varphi(x)|^{2}dx
-4e^{-2\lambda}\displaystyle\int_{{\bf R}^{3}}\Big[x\cdot (\nabla V)(e^{-2\lambda}x)\Big]|\varphi(x)|^{2}dx\nonumber\\
&\quad\quad\quad-9e^{6\lambda}\displaystyle\int_{{\bf R}^{3}}|\varphi(x)|^{4}dx\nonumber\\
&\quad\quad= 4s'(\lambda)-3\displaystyle\int_{{\bf R}^{3}}|\varphi(x)|^{4}dx\leq 4s'(\lambda)\nonumber,
\end{align}
where we have used the inequalities $x\cdot\nabla V$ and $3x\cdot\nabla V+x\nabla^{2} Vx^{T}\leq 0$.

1. If $\varphi\in {\mathcal{N}}^{-}$, then it follows from \eqref{s'} that $s'(0)=P_{k}(\varphi)<0$ and $s'(\lambda)>0$ for sufficiently small $\lambda<0$.
Thus, by the continuity of $s'(\lambda)=P_{k}(\varphi_{\lambda}^{3,-2})$ in $\lambda$, there exists a negative $\lambda_{0}<0$ such that
\begin{align*}
s'(\lambda_{0})=P_{k}(\varphi_{\lambda_{0}}^{3,-2})=0\;\;\text{and}\;\;s'(\lambda)<0,\;\;\text{for}\;\;\forall \lambda\in (\lambda_{0},0].
\end{align*}
By the definition of $n_{0}$ \eqref{threshold}, $s(\lambda_{0})=S_{k}(\varphi_{\lambda_{0}}^{3,-2})\geq n_{0}$. Integrating the inequality \eqref{s''} over
$[\lambda_{0}, 0]$ yields that
$$
P_{k}(\varphi)=s'(0)=s'(0)-s'(\lambda_{0})\leq 4(s(0)-s(\lambda_{0}))\leq 4(S_{k}(\varphi)-n_{0})=-4(n_{0}-S_{k}(\varphi)).
$$
Thus, we complete the proof of \eqref{upperbound}.

2. If $\varphi\in {\mathcal{N}}^{+}$, we consider two cases: one is $8P_{k}(\varphi)\geq 3\|\varphi\|_{L^{4}}^{4}$ and
the other is $8P_{k}(\varphi)< 3\|\varphi\|_{L^{4}}^{4}$.

For the case $8P_{k}(\varphi)\geq 3\|\varphi\|_{L^{4}}^{4}$, it follows from the definition of $P_{k}$ \eqref{functionalP} that
$$
2P_{k}(\varphi)=4\displaystyle\int_{{\bf R}^{3}}|\nabla \varphi(x)|^{2}dx
-2\displaystyle\int_{{\bf R}^{3}}(x\cdot\nabla V) |\varphi(x)|^{2}dx
-3\displaystyle\int_{{\bf R}^{3}}|\varphi(x)|^{4}dx,
$$
and then we have
$$
10P_{k}(\varphi)\geq 4(\|\nabla\varphi\|_{L^{2}}^{2}-\frac{1}{2}\displaystyle\int_{{\bf R}^{3}}(x\cdot\nabla V) |\varphi(x)|^{2}dx),
$$
that is,
\begin{align}\label{lowerbound1}
P_{k}(\varphi)\geq \frac{2}{5}\Big(\|\nabla\varphi\|_{L^{2}}^{2}-\frac{1}{2}\displaystyle\int_{{\bf R}^{3}}(x\cdot\nabla V) |\varphi(x)|^{2}dx\Big).
\end{align}
For the other case $8P_{k}(\varphi)< 3\|\varphi\|_{L^{4}}^{4}$, by \eqref{s''}, we have
\begin{align}\label{s'decreasing}
0<8s'(\lambda)<3e^{6\lambda}\|\varphi\|_{L^{4}}^{4}\;\;\text{and then}\;\;s''(\lambda)\leq 4s'(\lambda)-3e^{6\lambda}\|\varphi\|_{L^{4}}^{4}< -4s'(\lambda)
\end{align}
at $\lambda=0$. Also as $s'(\lambda)$ and $s''(\lambda)$ are continuous, $s'(\lambda)$ decreases as $s$ increases until $s'(\lambda_{1})=0$ for some $0<\lambda_{1}<+\infty$
and \eqref{s'decreasing} is true over $[0, \lambda_{1}]$. Since $P_{k}(\varphi_{\lambda_{1}}^{3,-2})=s'(\lambda_{1})=0$, by
the definition of $n_{0}$ \eqref{threshold}, $s(\lambda_{1})=S_{k}(\varphi_{\lambda_{1}}^{3,-2})\geq n_{0}$. Integrating the second inequality in \eqref{s'decreasing} over
$[0, \lambda_{1}]$, we have
\begin{align*}
-P_{k}(\varphi)=s'(\lambda_{1})-s'(0)<-4(s(\lambda_{1})-s(0))\leq -4\Big(n_{0}-S_{k}(\varphi)\Big),
\end{align*}
which is
\begin{align}\label{lowerbound2}
P_{k}(\varphi)\geq 4\Big(n_{0}-S_{k}(\varphi)\Big).
\end{align}
Putting \eqref{lowerbound1} and \eqref{lowerbound2} together yields \eqref{lowerbound}.

\end{proof}

\section{Criteria for global well-posedness and blow-up}
\setcounter{equation}{0}

In this section, we will give the criteria for global well-posedness and blow-up for the solution $u$ of ($\rm{NLS_{k}}$), which are partial results of Theorem \ref{scattering1}.
The proof of blow-up part is based on the argument of \cite{DWZ}.

\begin{theorem}\label{globalvsblowup}
Let $u$ be the solution of ($\rm{NLS_{k}}$) on $(-T_{min}, T_{max})$, where $(-T_{min}, T_{max})$ is the maximal life-span.

(i) If $u_{0}\in {\mathcal{N}}^{+}$, then $u$ is global well-posedness and $u(t)\in {\mathcal{N}}^{+}$ for any $t\in {\bf R}$.

(ii) If $u_{0}\in {\mathcal{N}}^{-}$, then $u(t)\in {\mathcal{N}}^{-}$ for any $t\in (-T_{min}, T_{max})$ and one of the following four statements holds true:

\;\;(a) $T_{max}<+\infty$ and $\lim_{t\uparrow T_{max}}\|\nabla u(t)\|_{L^{2}}=+\infty$.

\;\;(b) $T_{min}<+\infty$ and $\lim_{t\downarrow -T_{min}}\|\nabla u(t)\|_{L^{2}}=+\infty$.

\;\;(c) $T_{max}=+\infty$ and there exists a sequence $\{t_{n}\}_{n=1}^{+\infty}$ such that $t_{n}\rightarrow+\infty$ and $\lim_{t_{n}\uparrow +\infty}\|\nabla u(t)\|_{L^{2}}=+\infty$.

\;\;(d) $T_{min}=+\infty$ and there exists a sequence $\{t_{n}\}_{n=1}^{+\infty}$ such that $t_{n}\rightarrow-\infty$ and $\lim_{t_{n}\downarrow -\infty}\|\nabla u(t)\|_{L^{2}}=+\infty$.
\end{theorem}

\begin{proof}

$(i)$ Define
$$
I^{+}=\{t\in (-T_{min}, T_{max}): u(t)\in {\mathcal{N}}^{+}\}.
$$
Obviously, $0\in I^{+}\neq \Phi$. On one hand, since $S_{k}(u(t))=S_{k}(u_{0})<n_{0}$ and $P_{k}(u(t))$ is continuous in $t$, $I^{+}$ is a closed subset of $(-T_{min}, T_{max})$.
On the other hand, by \eqref{lowerbound}, we further obtain that, $I^{+}$ is a open subset of $(-T_{min}, T_{max})$. Therefore, $I^{+}=(-T_{min}, T_{max})$. Namely,
for any $t\in (-T_{min}, T_{max})$, $u(t)\in {\mathcal{N}}^{+}$. It follows form \eqref{equivalenceSH} that for any $t\in (-T_{min}, T_{max})$,
$$
\|u(t)\|_{H^{1}}^{2}\leq \|u(t)\|_{{\mathcal{H}}_{k}^{1}}^{2}\leq 4S_{k}(u(t))\leq 4n_{0}.
$$
So by local well-posedness result $(i)$ of Theorem \ref{localwellposedness}, we have $(-T_{min}, T_{max})={\bf R}$, which implies that
$u$ is global well-posedness and $u(t)\in {\mathcal{N}}^{+}$ for any $t\in {\bf R}$. Thus, we complete the proof of $(i)$.

$(ii)$ Similarly above, we can show that $u(t)\in {\mathcal{N}}^{-}$ for any $t\in (-T_{min}, T_{max})$ by replacing \eqref{lowerbound} with \eqref{upperbound}.

In the sequel, we only consider the positive time because the negative time can be dealt with similarly. If $T_{max}<+\infty$, we naturally have
$\lim_{t\uparrow T_{max}}\|\nabla u(t)\|_{L^{2}}=+\infty$. If $T_{max}=+\infty$, we shall prove $\lim_{t\uparrow +\infty}\|\nabla u(t)\|_{L^{2}}=+\infty$
by contradiction.
Assume we have
$$
C_0=\sup_{t\in\mathbb R^+}\|\nabla u(t)\|_{L^2}<\infty.
$$
Consider the localized Virial identity and define
\begin{equation}\label{vf}
I(t):=\int_{{\bf R}^{3}}\phi(x)|u(t,x)|^2dx,
\end{equation}
then by straight computations, we obtain that for any $\phi\in C^4(\mathbb R^3)$ (e.g., see Proposition 7.1 in \cite{H})
\begin{equation*}
I'(t)=2\Im\int_{{\bf R}^{3}}\nabla\phi\cdot\nabla u\bar udx;
\end{equation*}
\begin{align*}
I''(t)=\int_{{\bf R}^{3}}4\Re\nabla\bar u\nabla^2\phi\nabla udx
-\int_{{\bf R}^{3}}2\nabla\phi\cdot\nabla V|u|^2+\Delta\phi |u|^4dx
-\int_{{\bf R}^{3}}\Delta^2\phi|u|^2dx.
\end{align*}
In particular, if $\phi$ is a radial function ,
\begin{align}\label{I'}
I'(t)=2\Im\int_{{\bf R}^{3}}\phi'(r)\frac{x\cdot\nabla u}r\bar udx,
\end{align}
\begin{align}\label{I''0}
&I''(t)=4\int_{{\bf R}^{3}}\frac{\phi'}r|\nabla u|^2dx+4\int_{{\bf R}^{3}}\left(\frac{\phi''}{r^2}-\frac{\phi'}{r^3}\right)|x\cdot\nabla u|^2dx\\
&-2\int_{{\bf R}^{3}}\frac{\phi'}{r}x\cdot\nabla V|u|^2dx-\int_{{\bf R}^{3}}\left(\phi''(r)+\frac{2}r\phi'(r)\right) |u|^4dx\nonumber
-\int_{{\bf R}^{3}}\Delta^2\phi|u|^2dx.
\end{align}

{\bf $L^2$ estimate in the exterior ball}

Given $R\gg 1$, which will be determined later. Take $\phi$ in \eqref{vf} such  that
$$
\phi=\begin{cases}0,&0\leq r\leq\frac R2;\\1,&r\geq R,\end{cases}
$$
and
$$
0\leq\phi\leq1,\ \ 0\leq\phi'\leq\frac4R.
$$
By \eqref{I'} and H\"{o}lder inequality, there holds that
\begin{align*}
I(t)=&I(0)+\int_0^tI'(\tau)d\tau
\leq I(0)+t\|\phi'\|_{L^\infty}M(u_{0}) C_0\\
\leq&\int_{|x|\geq\frac R2}|u_0|^2dx+\frac{4M(u_{0}) C_0t}R.
\end{align*}
Note that
$$
\int_{|x|\geq\frac R2}|u_0|^2dx=o_R(1),
$$
and
$$
\int_{|x|\geq  R}|u(t,x)|^2dx\leq I(t).
$$
So for given $\eta_0>0$, if
$$
t\leq\frac{\eta_0R}{4M(u_{0}) C_0},
$$
then we have that
\begin{align}\label{outermass}
\int_{|x|\geq\frac R2}|u(t,x)|^2dx\leq\eta_0+o_R(1).
\end{align}

{\bf Localized Virial identity}

$I''(t)$ can be rewritten as
\begin{align}\label{I''}
I''(t)=4P_{k}(u)+R_1+R_2+R_3+R_4,
\end{align}
where
$$R_1=4\int_{{\bf R}^{3}}\left(\frac{\phi'}r-2\right)|\nabla u|^2dx+4\int_{{\bf R}^{3}}
\left(\frac{\phi''}{r^2}-\frac{\phi'}{r^3}\right)|x\cdot\nabla u|^2dx,$$
$$R_2=-\int_{{\bf R}^{3}}\left(\phi''+\frac{2}r\phi'(r)-6\right)| u|^4dx,$$
$$R_3=-2\int_{{\bf R}^{3}}\left(\frac{\phi'}{r}-2\right)(x\cdot\nabla V)|u|^2dx,$$
$$R_4=-\int_{{\bf R}^{3}}\Delta^2\phi|u|^2dx.$$
At this stage, we choose another radial function $\phi$ such that
\begin{align}\label{radialfunction}
0\leq\phi\leq r^2,\ \ \phi''\leq2,\ \ \phi^{(4)}\leq\frac4{R^2},\;\;
\text{and}\;\;
\phi=\begin{cases}r^2,&0\leq r\leq R;\\0,&r\geq 2R\end{cases}.
\end{align}
First, note that $R_1\leq0$.
If $\phi''\leq r^{-1}\phi'\leq0$, by $\phi'\leq2r$, it is easy to see that
$R_1\leq0$.
If $\phi''\leq r^{-1}\phi'\leq0$, by $\phi''\leq2$, it holds that
$$
R_1\leq 4\int_{{\bf R}^{3}}\left(\frac{\phi'}r-2\right)|\nabla u|^2dx+4\int_{{\bf R}^{3}}
\left(\phi''-\frac{\phi'}r\right)|\nabla u|^2dx=4\int_{{\bf R}^{3}}\left(\phi''-2\right)|\nabla u|^2dx\leq 0.
$$
Secondly, let $\chi^{4}(r)=|\phi''+\frac{2}r\phi'(r)-6|$. it is easy to see that ${\rm supp}\chi\subset[R,\infty).$
So by Gagliardo-Nirenberg inequality
$$
R_2\leq \displaystyle\int_{{\bf R}^{3}}\chi^{4}|u|^{4}dx\lesssim \|\nabla(\chi u)\|_{L^{2}}^{3}\|\chi u\|_{L^{2}}
\leq C(M(u_{0}), C_0) \|u\|_{L^2(|x|>R)}^{\frac{1}{4}}.
$$
By the properties of $\phi$,
$$
R_4\leq CR^{-2}\|u\|_{L^2(|x|>R)}^2.
$$
Finally, by $x\cdot\nabla V\leq0$, and we obtain $R_3\leq0$.

Putting all the above estimates together, there holds that for $R\gg 1$,
\begin{align}\label{I''upperbound}
I''(t)\leq 4P_{k}(u)+\tilde C\|u\|_{L^2(|x|>R)},
\end{align}
where $\tilde C>0$ depending on $M(u_{0})$ and $C_0$.

Using \eqref{outermass} and \eqref{I''upperbound} yields that
$t\leq T:=\eta_0R/(4M(u_{0}) C_0)$,
$$I''(t)\leq 4P_{k}(u)+\tilde C(\eta_0^{1/2}+o_R(1)).$$
As $u(t)\in {\mathcal{N}}^{-}$, it follows from \eqref{upperbound} that there exists
$$\beta_{0}:=-4(n_{0}-S_{k}(u(t)))=-4(n_{0}-S_{k}(u_{0}))<0$$ such that
\begin{align}\label{I''upperbound1}
I''(t)\leq 4\beta_{0}+\tilde C(\eta_0^{1/2}+o_R(1)).
\end{align}
Choose $\eta_0$ sufficiently small
and take $R$ sufficiently large such that
$$
4\beta_{0}+\tilde C\eta_0^{1/2}+o_R(1)<\beta_{0}.
$$
Integrating \eqref{I''upperbound1} over $[0, T]$ twice,
 we obtain that
\begin{align*}
I(T)&\leq I(0)+I'(0)T+\int_0^T\int_0^t\beta_{0}\\
&\leq I(0)+I'(0)T+\beta_{0}\frac{T^2}2.
\end{align*}
Hence for $T=\eta_0R/(4M(u_{0})C_0)$, we obtain that
\begin{align}\label{IT}
I(T)\leq I(0)+I'(0)\eta_0R/(4M(u_{0}) C_0)+\alpha_0R^2,
\end{align}
where the constant $$\alpha_0=\beta_0\eta_0^2/(4M(u_{0}) C_0)^2<0$$ is independent of $R$.

At the same time, we note that
\begin{align}\label{I0}
I(0)=o_R(1)R^2,\ \ \ I'(0)=o_R(1)R.
\end{align}
In fact,
\begin{align*}
I(0)&\leq\int_{|x|<\sqrt{R}}|x|^2|u_0|^2dx+\int_{\sqrt{R}<|x|<2R}|x|^2|u_0|^2dx\\
&\leq RM(u_{0})+R^2\int_{|x|>\sqrt{R}}|u_0|^2dx=o_R(1)R^2.
\end{align*}
Similarly, we obtain the second estimate and then prove \eqref{I0}.

Putting \eqref{IT} with \eqref{I0} together and choosing $R$ sufficiently enough, we find that
$$
I(T)\leq (o_R(1)+\alpha_0)R^2\leq\frac14\alpha_0R^2<0,
$$
which is impossible since $I\geq 0$. Thus,
we conclude the proof of blow-up part.

\end{proof}

\medskip

\section{Scattering result}
\setcounter{equation}{0}
In this section, we shall show the remaining part of Theorem \ref{scattering1}. In the previous section, we have obtained
that the solution $u(t)$ of ($\rm{NLS_{k}}$) is global and belongs to $\mathcal{N}^{+}$ if $u_{0}\in \mathcal{N}^{+}$.
To get scattering result, by $(iv)$ of
Theorem \ref{localwellposedness}, it's enough to get \eqref{scatteringbound}. To this end, we introduce a definition.

\begin{definition}\label{SC}
We say that $SC(u_0)$ holds if for $u_0\in H^1({\bf R}^{3})$ satisfying
$u_{0}\in \mathcal{N}^{+}$,
 the corresponding global solution $u$  of ($\rm{NLS_{k}}$)  satisfies \eqref{scatteringbound}.
\end{definition}
We first note that for $u_{0}\in \mathcal{N}^{+}$, there exists $\delta>0$ such that if $S_{k}(u_{0})<\delta$, then
\eqref{scatteringbound} holds. In fact, by \eqref{equivalenceSH}, $\|u_{0}\|_{H^{1}}\lesssim S_{k}(u_{0})$. Therefore, by $(ii)$
of Theorem \ref{localwellposedness}, taking $\delta>0$ sufficiently small gives \eqref{scatteringbound}.

Now for each $\delta>0$, we define the set $S_\delta$ as follows:
\begin{align}\label{Sdelta}
S_\delta=\{u_0\in H^1({\bf R}^{3}):\ \ S_{k}(u_{0})<\delta \ \  and \ \
u_{0}\in{\mathcal{N}^{+}} \Rightarrow \ \eqref{scatteringbound}\ holds\}.
\end{align}
We also define
\begin{align}\label{nc}
n_c=\sup\{\delta:\ \ u_0\in S_\delta\Rightarrow SC(u_0)\ \  holds \}.
\end{align}
Hence, $0<n_{c}\leq n_{0}$. Next we shall prove that $n_{c}<n_{0}$ is impossible, which implies that $n_{c}=n_{0}$.
Thus, we assume now
$$
n_{c}<n_{0}.
$$
By the definition of $n_{c}$, we can find a sequence of solutions $u_{n}$ of ($\rm{NLS_{k}}$) with initial data $\phi_{n}\in {\mathcal{N}^{+}}$
such that $S_{k}(\phi_{n})\rightarrow n_{c}$ and
\begin{align}\label{L5}
\|u_{n}\|_{L_{{\bf R}^{+},x}^{5}}=+\infty\;\;\text{and}\;\;\|u_{n}\|_{L_{{\bf R}^{-},x}^{5}}=+\infty.
\end{align}
In the subsequent subsection, our goal is to prove the existence of critical element $u_{c}\in H^1({\bf R}^{3})$, which is a global solution of
($\rm{NLS_{k}}$) with initial data $u_{c,0}$ such that $S_{k}(u_{c,0})=n_{c}$, $u_{c,0}\in {\mathcal{N}^{+}}$ and $SC(u_{c,0})$ does not hold. Moreover, we prove that
if $\|u_{c}\|_{L_{t,x}^{5}}=+\infty$, then $K:=\{u_c(t): t\in {\bf R}\}$ is precompact in  $H^1({\bf R}^{3})$.

Before showing the existence and compactness of critical element $u_{c}$, we need a lemma related with the linear profile decomposition Lemma \ref{linearprofile}.

\begin{lemma}\label{whole-partial}
Let $M\in {\bf N}$ and $\psi^{j}\in H^{1}({\bf R}^{3})$ for any $0\leq j\leq M$. Suppose that there exist some
$\delta>0$ and $\epsilon>0$ with $2\epsilon<\delta$ such that
\begin{align*}
\sum_{j=0}^{M}S_{k}(\psi^{j})-\epsilon\leq S_{k}\Big(\sum_{j=0}^{M}\psi^{j}\Big)\leq n_{0}-\delta,\:\:\;\;
-\epsilon\leq I_{k}\Big(\sum_{j=0}^{M}\psi^{j}\Big)\leq \sum_{j=0}^{M}I_{k}(\psi^{j})+\epsilon.
\end{align*}
Then $\psi^{j}\in \mathcal{N}^{+}$ for any $0\leq j\leq M$.
\end{lemma}

\begin{proof}
Suppose that for some $0\leq l\leq M$, $I_{k}(\psi^{l})<0$. By \eqref{positiveK} with $(a, b)=(3,0)$, we have
$I_{k}\Big((\psi^{l})_{\lambda}^{3,0}\Big)>0$ for sufficiently small $\lambda<0$. Thus, by continuity of $I_{k}\Big((\psi^{l})_{\lambda}^{3,0}\Big)$ in $\lambda$,
there exists $\lambda_{2}<0$ such that $I_{k}\Big((\psi^{l})_{\lambda_{2}}^{3,0}\Big)=0$. As $n_{k}^{3.0}=n_{0}$, by the increasing property of
$J_{k}^{3,0}\Big((\psi^{l})_{\lambda}^{3,0}\Big)$ in $\lambda$, we have
$$
J_{k}^{3,0}(\psi^{l})\geq J_{k}^{3,0}\Big((\psi^{l})_{\lambda_{2}}^{3,0}\Big)=S_{k}\Big((\psi^{l})_{\lambda_{2}}^{3,0}\Big)\geq n_{0}.
$$
By the nonnegativity of $J_{k}^{3,0}(\psi^{j})$ for any $0\leq j\leq M$ and $2\epsilon<\delta$, we have
\begin{align*}
n_{0}&\leq J_{k}^{3,0}(\psi^{l})\leq \sum_{j=0}^{M}J_{k}^{3,0}(\psi^{j})=\sum_{j=0}^{M}\Big(S_{k}(\psi^{j})-\frac{1}{2}I_{k}(\psi^{j})\Big)\\
&\leq S_{k}\Big(\sum_{j=0}^{M}\psi^{j}\Big)+\epsilon-\frac{1}{2}I_{k}\Big(\sum_{j=0}^{M}\psi^{j}\Big)+\frac{1}{2}\epsilon\\
&\leq n_{0}-\delta+2\epsilon<n_{0},
\end{align*}
which is impossible. Hence, for each $0\leq j\leq M$, we obtain
$$
I_{k}(\psi^{j})\geq 0.
$$
So
$$
S_{k}(\psi^{j})=J_{k}^{3,0}(\psi^{j})+\frac{1}{2}I_{k}(\psi^{j})\geq 0,
$$
which together with
$$
\sum_{j=0}^{M}S_{k}(\psi^{j})\leq S_{k}\Big(\sum_{j=0}^{M}\psi^{j}\Big)+\epsilon\leq n_{0}-\delta+\epsilon<n_{0}
$$
yields that $S_{k}(\psi^{j})<n_{0}$ for each $0\leq j\leq M$. By Lemma \ref{independentofab}, $\psi^{j}\in \mathcal{N}^{+}$ for any $0\leq j\leq M$.

\end{proof}

\subsection{Existence and compactness of critical element}

\begin{proposition}\label{criticalelement}
There exists a $u_{c,0}$ in $H^1({\bf R}^{3})$ with
$S_{k}(u_{c,0})=n_{c}$, $u_{c,0}\in {\mathcal{N}^{+}}$
such that if $u_c$ is the corresponding global solution of ($\rm{NLS_{k}}$) with
the initial data $u_{c,0}$, then
$\|u_c\|_{L_{t,x}^{5}({\bf R}\times{\bf R}^{3})}=+\infty$ and $K$ is precompact in  $H^1({\bf R}^{3})$.
\end{proposition}

\begin{proof}
We first note that $\{\phi_{n}\}_{n=1}^{+\infty}$ be a uniformly bounded sequence in $H^{1}({\bf R}^{3})$. In fact,
since $\phi_{n}\in \mathcal{N}^{+}$ for any $n\in {\bf N}$, by \eqref{equivalenceSH},
$$
\|\phi_{n}\|_{H^{1}}\leq \|\phi_{n}\|_{{\mathcal{H}}_{k}^{1}}^{2}\leq 4S_{k}(\phi_{n})<4n_{0}.
$$
We apply Lemma \ref{linearprofile} to $\phi_{n}$ to get that for each $M\leq M^{*}$,
\begin{align}\label{M}
\phi_{n}=\sum_{j=1}^{M}\psi_{n}^{j}+W_{n}^{M},
\end{align}
\begin{align*}
S_{k}(\phi_{n})=\sum_{j=1}^{M}S_{k}(\psi_{n}^{j})+S_{k}(W_{n}^{M})+o_{n}(1),
\end{align*}
and
\begin{align*}
I_{k}(\phi_{n})=\sum_{j=1}^{M}I_{k}(\psi_{n}^{j})+I_{k}(W_{n}^{M})+o_{n}(1),
\end{align*}
which together with $\phi_{n}\in \mathcal{N}^{+}$ yield that there exist some
$\delta>0$ and $\epsilon>0$ with $2\epsilon<\delta$ such that
\begin{align*}
\sum_{j=1}^{M}S_{k}(\psi_{n}^{j})+S_{k}(W_{n}^{M})-\epsilon\leq S_{k}(\phi_{n})\leq n_{0}-\delta,\:\:\;\;
-\epsilon\leq I_{k}(\phi_{n})\leq \sum_{j=1}^{M}I_{k}(\psi_{n}^{j})+I_{k}(W_{n}^{M})+\epsilon.
\end{align*}
According to Lemma \ref{whole-partial}, we have that for large $n$ and each $1\leq j\leq M$,
$\psi_{n}^{j}$, $W_{n}^{M}\in {\mathcal{N}^{+}}$, and then
\begin{align}\label{M=1}
0\leq \overline{\lim_{n\rightarrow+\infty}}S_{k}(\psi_{n}^{j})
\leq \overline{\lim_{n\rightarrow+\infty}}S_{k}(\phi_{n})=n_{c},
\end{align}
where if equality holds in the last inequality for some $j$, we must have $M^{*}=1$ and $W_{n}^{1}\rightarrow 0$ in $H^{1}$.

We claim that if equality holds in the last inequality of \eqref{M=1} for some $j$ (w.l.g. let $j=1$), $u_{c,0}$ is namely $\psi^{1}$.
Indeed, at this time we have
\begin{align}\label{psi1}
\phi_{n}=\psi_{n}^{1}+W_{n}^{1}=e^{it_{n}^{1}H_{\alpha}}\tau_{x_{n}^{1}}\psi^{1}+W_{n}^{1},
\end{align}
\begin{align*}
\overline{\lim_{n\rightarrow+\infty}}S_{k}(\psi_{n}^{1})=n_{c}
\end{align*}
and
\begin{align}\label{wn1}
W_{n}^{1}\rightarrow 0 \;\;\text{in}\;\; H^{1}.
\end{align}
Our target is to prove that
\begin{align}\label{xt}
x_{n}^{1}\equiv 0\;\;\text{ and}\;\; t_{n}^{1}\equiv 0.
\end{align}
If \eqref{xt} is true, then we have
\begin{align*}
\phi_{n}=\psi^{1}+W_{n}^{1},\;\;
S_{k}(\psi^{1})=n_{c},\;\; \psi^{1}\in\mathcal{N}^{+}
\end{align*}
and
\begin{align*}
\lim_{n\rightarrow+\infty}\|\phi_{n}-\psi^{1}\|_{H^{1}}=0.
\end{align*}
Here $\psi^{1}$ is namely our required $u_{c,0}$. Let $u_c$ be the solution of ($\rm{NLS_{k}}$) with
the initial data $u_{c,0}=\psi^{1}$, then $u_{c}$ is global and $S_{k}(u_{c})=S_{k}(u_{c,0})=n_{c}$.
Using Lemma \ref{stability}, it holds that $\|u_c\|_{L_{t,x}^{5}({\bf R}\times{\bf R}^{3})}=+\infty$. Otherwise,
$\|u_n\|_{L_{t,x}^{5}}<+\infty$, which contradicts with \eqref{L5}.

\noindent If \eqref{xt} is false, then either $|x_{n}^{1}|\rightarrow+\infty$ or $t_{n}^{1}\rightarrow\pm\infty$.
I will see that it leads to  $\|u_n\|_{L_{t,x}^{5}}<+\infty$ or $\|u_{n}\|_{L_{{\bf R}^{\pm},x}^{5}}<+\infty$
contradicting with \eqref{L5}.

\noindent For the case $|x_{n}^{1}|\rightarrow+\infty$, by \eqref{fractionaloperatorlimit}, we have
\begin{align}\label{H1}
\lim_{n\rightarrow +\infty}\|\psi_{n}^{1}\|_{{\mathcal H}_{k}^{1}}=\|\psi^{1}\|_{H^{1}}>0,
\end{align}
which implies that when $t_{n}^{1}\equiv 0$,
$$
S_{0}(\psi^{1})=\overline{\lim_{n\rightarrow+\infty}}S_{k}(\psi_{n}^{1})=n_{c}<n_{0}\;\;\text{and}
\;\;I_{0}(\psi^{1})=\overline{\lim_{n\rightarrow+\infty}}I_{k}(\psi_{n}^{1})\geq 0.
$$
By \eqref{threshold4}, $P_{0}(\psi^{1})\geq 0$. Hence, when $t_{n}^{1}\equiv 0$, $\psi^{1}$ satisfies
the condition \eqref{time0}. When $t_{n}^{1}\rightarrow\pm\infty$, apply \eqref{H1} and \eqref{decay6} to get
$$
\frac{1}{2}\|\psi^{1}\|_{H^{1}}^{2}=\overline{\lim_{n\rightarrow+\infty}}S_{k}(\psi_{n}^{1})=n_{c}<n_{0},
$$
that is, $\psi^{1}$ satisfies the condition \eqref{timeinfty}. Using Theorem \ref{nonlinearprofile} yields that
the solution $NLS_{k}(t)\psi_{n}^{1}$ of ($\rm{NLS_{k}}$) with initial data $\psi_{n}^{1}$ is global and satisfies
$$
\|NLS_{k}(t)\psi_{n}^{1}\|_{S_{\alpha}^{1}(I)}\lesssim_{\|\psi^{1}\|_{H^{1}}}1.
$$
We know that $W_{n}^{1}\rightarrow 0$ in $H^{1}$, which is
\begin{align*}
\lim_{n\rightarrow+\infty}\|\phi_{n}-\psi_{n}^{1}\|_{H^{1}}=0.
\end{align*}
Using Lemma \ref{stability} again, we obtain $\|u_n\|_{L_{t,x}^{5}}<+\infty$.

\noindent For the other case $t_{n}^{1}\rightarrow\pm\infty$, we only cope with $t_{n}^{1}\rightarrow-\infty$ since
$t_{n}^{1}\rightarrow+\infty$ can be dealt with similarly. Apply \eqref{psi1} with $x_{n}^{1}\equiv 0$,
\eqref{wn1}, Strichartz estimates Lemma \ref{Strichartz} and norm equivalence Lemma \ref{Sobolev} to get that
\begin{align*}
\lim_{n\rightarrow+\infty}\|e^{-itH_{\alpha}}\phi_{n}\|_{L_{{\bf R}^{+},x}^{5}}
&\leq \lim_{n\rightarrow+\infty}\|e^{-i(t-t_{n}^{1})H_{\alpha}}\psi^{1}\|_{L_{{\bf R}^{+},x}^{5}}
+\lim_{n\rightarrow+\infty}\|e^{-itH_{\alpha}}W_{n}^{1}\|_{L_{{\bf R}^{+},x}^{5}}\\
&\lesssim \lim_{n\rightarrow+\infty}\|W_{n}^{1}\|_{H^{1}}
+\lim_{n\rightarrow+\infty}\|e^{-itH_{\alpha}}\psi^{1}\|_{L_{(-t_{n}^{1}, +\infty),x}^{5}}=0,
\end{align*}
which immediately implies that $\lim_{n\rightarrow+\infty}\|u_{n}\|_{L_{{\bf R}^{+},x}^{5}}=0$ by $(ii)$ of Theorem
\ref{localwellposedness}. Thus, we obtain that \eqref{xt} holds true.

Next we turn to the other situation that equality  doesn't hold in the last inequality of \eqref{M=1} for any $1\leq j\leq M$.
So for each $1\leq j\leq M$ and $\psi_{n}^{j}\in \mathcal{N}^{+}$, there exists $\delta=\delta_{j}>0$ such that
\begin{align*}
\overline{\lim_{n\rightarrow+\infty}}S_{k}(\psi_{n}^{j})
\leq n_{c}-2\delta,\;\; P_{k}(\psi_{n}^{j})\geq 0\;\;\text{and}\;\;I_{k}(\psi_{n}^{j})\geq 0.
\end{align*}
We shall use $\psi_{n}^{j}$ to constitute approximate solutions of $u_{n}$ under three cases: $|x_{n}^{j}|\rightarrow+\infty$;
$x_{n}^{j}\equiv 0$ and $t_{n}^{j}\equiv 0$; $x_{n}^{j}\equiv 0$ and $t_{n}^{j}\rightarrow\pm\infty$ and then apply Lemma \ref{stability}
to get a contradiction.

For some $j$ such that $|x_{n}^{j}|\rightarrow+\infty$, \eqref{H1} still holds for $\psi_{n}^{j}$. Using the same argument after \eqref{H1},
we obtain that $\psi^{j}$ satisfies \eqref{time0} or \eqref{timeinfty}. Therefore, using Theorem \ref{nonlinearprofile}, we can constitute
a global solution $v_{n}^{j}(t):=NLS_{k}(t)\psi_{n}^{j}$ of ($\rm{NLS_{k}}$) with initial data $\psi_{n}^{j}$ such that
$$
\|v_{n}^{j}\|_{L_{t,x}^{5}}\leq \|NLS_{k}(t)\psi_{n}^{j}\|_{S_{\alpha}^{1}(I)}\lesssim_{\|\psi^{j}\|_{H^{1}}}1.
$$

For some $j$ such that $x_{n}^{j}\equiv 0$ and $t_{n}^{j}\equiv 0$, we apply $\psi^{j}\in \mathcal{N}^{+}$ to constitute
a global solution $v_{n}^{j}(t):=NLS_{k}(t)\psi^{j}$ of ($\rm{NLS_{k}}$) with initial data $\psi^{j}$.

For some $j$ such that $x_{n}^{j}\equiv 0$ and $t_{n}^{j}\rightarrow\pm\infty$, by $(iii)$ of Theorem \ref{localwellposedness}, there exists
$\tilde{\psi}^{j}\in H^{1}$ such that
\begin{align}\label{initialdata}
\|NLS_{k}(t_{n}^{j})\tilde{\psi}^{j}-e^{it_{n}^{j}H_{\alpha}}\psi^{j}\|_{{\mathcal{H}}_{k}^{1}}
\sim \|NLS_{k}(t_{n}^{j})\tilde{\psi}^{j}-e^{it_{n}^{j}H_{\alpha}}\psi^{j}\|_{H^{1}}\rightarrow 0\;\;\text{as}\;\;n\rightarrow+\infty,
\end{align}
which implies that for each $1\leq j\leq M$ and $n$ large enough,
\begin{align*}
S_{k}\Big(NLS_{k}(t_{n}^{j})\tilde{\psi}^{j}\Big)
\leq n_{c}-\delta,\;\; P_{k}\Big(NLS_{k}(t_{n}^{j})\tilde{\psi}^{j}\Big)\geq 0\;\;\text{and}\;\;I_{k}\Big(NLS_{k}(t_{n}^{j})\tilde{\psi}^{j}\Big)\geq 0.
\end{align*}
We set $v_{n}^{j}(0)=NLS_{k}(t_{n}^{j})\tilde{\psi}^{j}$. Then according to the definition of $n_{c}$ and $v_{n}^{j}(0)\in\mathcal{N}^{+} $,
we obtain that the solution $v_{n}^{j}(t):=NLS_{k}(t+t_{n}^{j})\tilde{\psi}^{j}$ of ($\rm{NLS_{k}}$) with initial data $v_{n}^{j}(0)$ is global
and satisfies uniform space-time bounds: $\|v_{n}^{j}\|_{L_{t,x}^{5}}<+\infty$.

As a result, we can construct approximate solutions of ($\rm{NLS_{k}}$):
$$
\tilde{u}_{n}(t):=\sum_{j=1}^{M}v_{n}^{j}+e^{-itH_{\alpha}}W_{n}^{M}
$$
and set
$$
e:=(i\partial_{t}-H_{\alpha})\tilde{u}_{n}+|\tilde{u}_{n}|^{2}\tilde{u}_{n}.
$$
By \eqref{M} and \eqref{initialdata}, we have
\begin{align}\label{approximate1}
\|\phi_{n}-\tilde{u}_{n}(0)\|_{H^{1}}=\|u_{n}(0)-\tilde{u}_{n}(0)\|_{H^{1}}\rightarrow 0\;\;\text{as}\;\;n\rightarrow+\infty,
\end{align}
which implies that
\begin{align}\label{approximate2}
\overline{\lim}_{n\rightarrow+\infty}\|\tilde{u}_{n}(0)\|_{H^{1}}\;\;\text{ has a uniform bound independent of}\;\; M.
\end{align}
Using the same argument of Lemma 7.3 in \cite{KMVZ} and replacing $H_{\alpha}$ and homogeneous fractional operator (e.g.,
$|\nabla|^{\frac{1}{2}}$) with $1+H_{\alpha}$ and inhomogeneous fractional operator (e.g., $(1+\Delta)^{\frac{1}{4}}$), respectively,
we also obtain the same results as there for $\tilde{u}_{n}(t)$, that is,
\begin{align}\label{approximate3}
\overline{\lim}_{n\rightarrow+\infty}\|\tilde{u}_{n}\|_{L_{t,x}^{5}} \;\;\text{has a uniform bound independent of }\;\;M
\end{align}
and
\begin{align}\label{approximate4}
\lim_{M\rightarrow M^{*}}\overline{\lim}_{n\rightarrow+\infty}
\|(1+\Delta)^{\frac{1}{4}}e\|_{L_{t,x}^{\frac{10}{7}}}=0
\end{align}
Applying \eqref{approximate1}-\eqref{approximate4} to Theorem \ref{stability} gives
$\|u_n\|_{L_{t,x}^{5}}<+\infty$, which is a contradiction with \eqref{L5}. Thus, we have completed the proof of existence of critical
element $u_{c}$.

Finally, we consider precompactness of $K$ in $H^{1}$. We recall that $u_{c}$ satisfies the following properties:
\begin{align*}
S_{k}(u_{c}(t))=n_{c}\;\;,\;\; u_{c}(t)\in {\mathcal{N}^{+}}\;\;\text{ for}\;\; \forall t\in {\bf R}\;\;\text{ and}\;\;
\|u_c\|_{L_{t,x}^{5}}=+\infty.
\end{align*}
In particular, for any time sequence $\{t_{n}\}_{n=1}^{+\infty}$, the sequence $\{u_{c}(t_{n})\}_{n=1}^{+\infty}$ also satisfies
\begin{align*}
S_{k}(u_{c}(t_{n}))=n_{c}\;\;,\;\; u_{c}(t_{n})\in {\mathcal{N}^{+}}\;\;\text{ for}\;\; \forall t\in {\bf R}\;\;\text{ and}\;\;
\|u_c\|_{L_{(-\infty, t_{n}),x}^{5}}=\|u_c\|_{L_{( t_{n}, +\infty),x}^{5}}=+\infty.
\end{align*}
Hence, regarding $u_{c}(t_{n})$ as the foregoing $\phi_{n}$ and noting that the fact $\phi_{n}$ converges $\psi^{1}$ in $H^{1}$ yieds
that $K$ is precompact in $H^{1}$. Thus, we complete the whole proof.
\end{proof}

\subsection{Precluding the critical element}

In this subsection, we shall apply the localized Virial identities \eqref{I'} and \eqref{I''} to preclude the critical element $u_{c}$.
First by the precompactness of $K$, we have uniform localization of $u_{c}$: For each
$\epsilon>0,$ there exists $R=R(\epsilon)>0$ independent of $t$ such that
\begin{align}\label{localization}
\int_{|x|>R}\Big(|\nabla u_{c}(t,x)|^2+|u_{c}(t,x)|^2+|u_{c}(t,x)|^4\Big)dx\leq\epsilon.
\end{align}
We next claim that there exists a constant $c$ such that for any $t\in {\bf R}$,
\begin{align}\label{control}
\|\nabla u_{c}(t)\|_{L^{2}}\geq c\|u_{c}(t)\|_{L^{2}}.
\end{align}
Indeed, if it's false, then there exists a time sequence $\{t_{n}\}_{n=1}^{+\infty}$ such that
$$
\|\nabla u_{c}(t_{n})\|_{L^{2}}\leq \frac{1}{n}\|u_{c}(t_{n})\|_{L^{2}}=\frac{1}{n}\|u_{c}(0)\|_{L^{2}},
$$
which means that $u_{c}(t_{n}\rightarrow 0$ in $\dot{H}^{1}$. However, $\{u_{c}(t_{n})\}_{n=1}^{+\infty}$ is precompact in $H^{1}$.
Hence, there exists a subsequence (still denoted by itself) $u_{c}(t_{n})\rightarrow 0$ in $H^{1}$. As $u_{c}(t_{n})\in {\mathcal{N}^{+}}$,
by \eqref{equivalenceSH}, $n_{c}=\lim_{n\rightarrow+\infty}S_{k}(u_{c}(t_{n}))=0$, which is impossible.

Now we use localized Virial identities \eqref{I'} and \eqref{I''} only with $u_{c}$ in place of $u$ again, where we still choose the radial function $\phi$ satisfying
\eqref{radialfunction}. For $R_{1}$, $R_{2}$, $R_{3}$ and $R_{4}$ in \eqref{I''}, by \eqref{localization}, we have
\begin{align}\label{importanterror}
|R_{1}+R_{2}+R_{3}+R_{4}|\lesssim &\displaystyle\int_{|x|\geq R}\Big(|\nabla u_{c}(t,x)|^{2}+|u_{c}(t,x)|^{4}+\frac{1}{R^{2}}|u_{c}(t,x)|^{2}
+\frac{1}{R^{\alpha}}|u_{c}(t,x)|^{2}\Big)dx\nonumber\\
&\rightarrow 0\;\;\text{as}\;\; R\rightarrow+\infty.
\end{align}
And for $P_{k}(u_{c})$ in \eqref{I''}, by \eqref{lowerbound}, $x\cdot \nabla V\leq 0$, \eqref{control} and \eqref{equivalenceSH}, we have
\begin{align}\label{importantlowerbound}
P_{k}(u_{c}(t))&\geq \min\Big\{4\Big(n_{0}-S_{k}(u_{c}(t))\Big),
\frac{2}{5}\Big(\|\nabla u_{c}(t)\|_{L^{2}}^{2}-\frac{1}{2}\displaystyle\int_{{\bf R}^{3}}(x\cdot\nabla V) |u_{c}(t)|^{2}dx)\Big)\Big\}\nonumber\\
&\gtrsim \|\nabla u_{c}(t)\|_{L^{2}}^{2}\gtrsim \|u_{c}(t)\|_{H^{1}}^{2}\gtrsim S_{k}(u_{c}(t))=n_{c}.
\end{align}
It follows from \eqref {importanterror} and \eqref {importantlowerbound} that there exists $\delta_{0}>0$ such that for $R$ large enough,
$$
I''(t)\geq \delta_{0},
$$
which means that $\lim_{t\rightarrow+\infty}I'(t)=+\infty$. However, it is impossible since $I'(t)$ is bounded. In fact,
by \eqref{I'},
$$
|I'(t)|\lesssim R.
$$
Thus, we complete the proof of scattering part of Theorem \ref{scattering1}.

\end{document}